\documentclass[11pt]{article}
\usepackage{amsfonts, amsmath, amssymb, latexsym, eucal, amscd,cite} 
\usepackage[all]{xy}
\newtheorem{theorem}{Theorem}[section]
\newtheorem{lemma}[theorem]{Lemma}
\newtheorem{proposition}[theorem]{Proposition}
\newtheorem{corollary}[theorem]{Corollary}
\newtheorem{exAux}[theorem]{Example}

\newtheorem{Def}[theorem]{Definition}
\newenvironment{definition}{\begin{Def} \rm}{\end{Def}}
\newtheorem{Note}[theorem]{Note}
\newenvironment{note}{\begin{Note} \rm}{\end{Note}}
\newtheorem{Problem}[theorem]{Problem}
\newenvironment{problem}{\begin{Problem} \rm}{\end{Problem}}
\newtheorem{Rem}[theorem]{Remark}

\newtheorem{Not}[theorem]{Notation}

\newtheorem{Conj}[theorem]{Conjecture}

\newtheorem{Ass}[theorem]{Assumption}

\newenvironment{proof}{\medskip\noindent{\bf Proof.\ }}{\qed\medskip}
\newenvironment{proofof}[1]{\medskip\noindent{\bf Proof  of {#1}.\ 
}}{\qed\medskip}
\newcommand{\qed}{\hfill\mbox{$\Box$\qquad\qquad}}

\newcommand{\F}{\mathbb{F}}
\newcommand{\Mat}{\text{\rm Mat}}

\newcommand{\vphi}{\varphi}

\newif\ifDRAFT
\DRAFTfalse

\begin{document}

\title{Linear transformations that are tridiagonal with respect to
the three decompositions for an LR triple}

\author{Kazumasa Nomura}
\date{}
\maketitle

\bigskip

{
\small
\begin{quote}
\begin{center}
{\bf Abstract}
\end{center}
Recently, Paul Terwilliger introduced the notion of
a lowering-raising (or LR) triple,
and classified the LR triples.
An LR triple is defined as follows.
Fix an integer $d \geq 0$, a field $\F$,
and a vector space  $V$ over $\F$ with dimension $d+1$.
By a decomposition of $V$ we mean a sequence $\{V_i\}_{i=0}^d$ 
of $1$-dimensional subspaces of $V$ whose sum is $V$.
For a linear transformation $A$ from $V$ to $V$, 
we say $A$ lowers $\{V_i\}_{i=0}^d$ whenever
$A V_i = V_{i-1}$ for $0 \leq i \leq d$, where $V_{-1}=0$.
We say $A$ raises $\{V_i\}_{i=0}^d$ whenever
$A V_i = V_{i+1}$ for $0 \leq i \leq d$, where $V_{d+1}=0$.
An ordered pair of linear transformations $A,B$ from $V$ to $V$ is called LR
whenever there exists a decomposition $\{V_i\}_{i=0}^d$ of $V$ that is lowered by $A$ and raised by $B$.
In this case the decomposition $\{V_i\}_{i=0}^d$ is uniquely determined by $A,B$; 
we call it the $(A,B)$-decomposition of $V$.
Consider a $3$-tuple of linear transformations $A$, $B$, $C$ from $V$ to $V$ such that
any two of $A$, $B$, $C$ form an LR pair on $V$. Such a $3$-tuple is called
an LR triple on $V$.
Let $\alpha$, $\beta$, $\gamma$ be nonzero scalars in $\F$.
The triple $\alpha A, \beta B, \gamma C$ is an LR triple on $V$,
said to be associated to $A,B,C$.
Let $\{V_i\}_{i=0}^d$ be a decomposition of $V$ and let $X$ be a linear transformation
from $V$ to $V$.
We say $X$ is tridiagonal with respect to $\{V_i\}_{i=0}^d$
whenever $X V_i \subseteq V_{i-1} + V_i + V_{i+1}$ for $0 \leq i \leq d$.
Let $\cal X$ be the vector space over $\F$ consisting of 
the linear  transformations from $V$ to $V$ that are tridiagonal
with respect to the $(A,B)$ and $(B,C)$ and $(C,A)$ decompositions of $V$.
There is a special class of LR triples, called $q$-Weyl type.
In the present paper, we find a basis of $\cal X$ for each LR triple
that is not associated to an LR triple of $q$-Weyl type.
\end{quote}
}

\section{Introduction}
 
The equitable presentation for the quantum algebra $U_q(\mathfrak{sl}_2)$
was introduced in \cite{ITW} and further investigated in \cite{T:uqsl2, T:Billiard}.
For the lie algebra $\mathfrak{sl}_2$,
the equitable presentation was introduced in \cite{HarT} and comprehensively studied
in \cite{BenT}.
From the equitable point of view, consider a finite-dimensional irreducible module for $U_q(\mathfrak{sl}_2)$
or $\mathfrak{sl}_2$.
In \cite{BenT,T:uqsl2} three nilpotent linear transformations of the module are encountered,
with each transformation acting as a lowering map and raising map in multiple ways.
In order to describe this situation more precisely,
Paul Terwilliger  introduced the notion of a lowering-raising (or LR) triple
of linear transformations, and gave their complete classification
(see \cite{T:LRT}).

There are three decompositions associated with an LR triple.
In the present paper, we investigate the linear transformations 
that act in a tridiagonal manner on each of these three decompositions. 
In this section, we first recall the notion of an LR triple, and 
then state our main results.

Throughout the paper, fix an integer $d \geq 0$, a field $\F$,
and a vector space $V$ over $\F$ with dimension $d+1$.
Let $\text{End}(V)$ denote the $\F$-algebra consisting of the
$\F$-linear transformations from $V$ to $V$,
and let $\Mat_{d+1}(\F)$ denote the $\F$-algebra consisting of the
$(d+1) \times (d+1)$ matrices that have all entries in $\F$.
We index the rows and columns by $0,1,\ldots,d$.

By a {\em decomposition} of $V$ we mean a sequence $\{V_i\}_{i=0}^d$
of $1$-dimensional subspaces of $V$ such that $V = \sum_{i=0}^d V_i$
(direct sum).
Let $\{V_i\}_{i=0}^d$ be a decomposition of $V$.
For notational convenience define $V_{-1}=0$ and $V_{d+1}=0$.
For $A \in \text{End}(V)$, we say
{\em $A$ lowers $\{V_i\}_{i=0}^d$} whenever $A V_i = V_{i-1}$ 
for $0 \leq i \leq d$.
We say {\em $A$ raises $\{V_i\}_{i=0}^d$} whenever
$A V_i = V_{i+1}$ for $0 \leq i \leq d$.
An ordered pair $A,B$ of elements in $\text{End}(V)$ is called {\em LR} 
whenever these exists a decomposition of $V$ that is lowered by $A$
and raised by $B$.
In this case the decomposition $\{V_i\}_{i=0}^d$ is uniquely determined by $A,B$
(see \cite[Section 3]{T:LRT});
we call it the {\em $(A,B)$-decomposition} of $V$.
For $0 \leq i \leq d$ define $E_i \in \text{End}(V)$ such that
$(E_i - I) V_i = 0$ and $E_i V_j = 0$ for $0 \leq j \leq d$, $j \neq i$,
where $I$ denotes the identity in $\text{End}(V)$.
We have $E_i E_j = \delta_{i,j} E_i$ for $0 \leq i,j  \leq d$
and $I = \sum_{i=0}^d E_i$.
We call $\{E_i\}_{i=0}^d$ the {\em idempotent sequence} for $A,B$
(or $\{V_i\}_{i=0}^d$).

A $3$-tuple $A,B,C$ of elements in $\text{End}(V)$ is called an
{\em LR triple}
whenever any two of $A$, $B$, $C$ form an LR pair on $V$.
We say $A,B,C$ is {\em over} $\F$. We call $d$ the {\em diameter} of $A,B,C$.

Let $A,B,C$ be an LR triple on $V$ and let $A',B',C'$ be an LR triple
on a vector space $V'$ over $\F$ with dimension $d+1$.
By an {\em isomorphism of LR triples} from $A,B,C$ to $A',B',C'$ we mean
an $\F$-linear bijection $\sigma : V \to V'$ such that
$\sigma A = A' \sigma$, $\sigma B = B' \sigma$, $\sigma C = C' \sigma$.
The LR triples $A,B,C$ and $A',B',C'$ are said to be {\em isomorphic}
whenever there exists an isomorphism of LR triples from $A,B,C$
to $A',B',C'$.

Let $A,B,C$ be an LR triple on $V$.
Let $\{V_i\}_{i=0}^d$ (resp.\ $\{V'_i\}_{i=0}^d$) (resp.\ $\{V''_i\}_{i=0}^d$) be the
$(A,B)$-decomposition (resp.\ $(B,C)$-decomposition) (resp.\ $(B,C)$-decomposition)
of $V$,
and let $\{E_i\}_{i=0}^d$ (resp.\ $\{E'_i\}_{i=0}^d$) (resp.\ $\{E''_i\}_{i=0}^d$)
be the corresponding idempotent sequence.
We call the sequence
\begin{equation}
    (\{E_i\}_{i=0}^d, \{E'_i\}_{i=0}^d, \{E''_i\}_{i=0}^d)         \label{eq:idempotentdata}
\end{equation}
the {\em idempotent data} of $A,B,C$.
Define scalars
\begin{align*}
  a_i &= \text{tr} (C E_i), &
  a'_i &= \text{tr} (A E'_i), &
  a''_i &= \text{tr} (B E''_i) &
  && (0 \leq i \leq d),
\end{align*}
where tr means trace.
We call the sequence
\begin{equation}
   (\{a_i\}_{i=0}^d, \{a'_i\}_{i=0}^d, \{a''_i\}_{i=0}^d)          \label{eq:tracedata}
\end{equation}
the {\em trace data} of $A,B,C$.
The LR triple is said to be {\em bipartite} whenever each of $a_i$, $a'_i$, $a''_i$ is zero
for $0 \leq i \leq d$. 
In this case, $d$ is even (see \cite[Lemma 16.6]{T:LRT}); set $d=2m$.
The elements
$\sum_{j=0}^{m} E_{2j}$,
$\sum_{j=0}^{m} E'_{2j}$,
$ \sum_{j=0}^{m} E''_{2j}$
are equal (see \cite[Lemma 16.12]{T:LRT}).
We denote this element by $J$: 
\begin{equation}
 J = \sum_{j=0}^{m} E_{2j} 
   = \sum_{j=0}^{m} E'_{2j} = \sum_{j=0}^{m} E''_{2j}.    \label{eq:defJ}
\end{equation}
Observe
\begin{equation}
I-J =  \sum_{j=0}^{m-1} E_{2j+1}
  =  \sum_{j=0}^{m-1} E'_{2j+1} =  \sum_{j=0}^{m-1} E''_{2j+1}.  \label{eq:I-J}
\end{equation}

Let $A,B,C$ be an LR triple on $V$,
and $\alpha$, $\beta$, $\gamma$ be nonzero scalars in $\F$.
Then $\alpha A, \beta B, \gamma C$ is an LR triple on $V$,
and this LR triple has the same idempotent data as $A,B,C$
(see \cite[Lemma 13.22]{T:LRT}).
Two LR triples $A,B,C$ and $A',B',C'$ over $\F$ are said to be {\em associated}
whenever there exist nonzero scalars $\alpha$, $\beta$, $\gamma$
in $\F$ such that $A' = \alpha A$, $B' = \beta B$, $C' = \gamma C$.

There is a special class of LR triples, said to have $q$-Weyl type.
This is described as follows.
Let $0 \neq q \in \F$.
An LR pair $A,B$ on $V$ is said to have {\em $q$-Weyl type} whenever
$q^2 \neq 1$ and
\[
  \frac{ q AB - q^{-1} BA }
         { q-q^{-1}  }
  = I.
\]
In this case, $d \geq 1$ and $q$ is a $(2d+2)$-root of unity
 (see \cite[Lemma 4.16]{T:LRT}).
An LR triple $A,B,C$ on $V$ is said to have {\em $q$-Weyl type} whenever
the LR pairs $A,B$ and $B,C$ and $C,A$ all have $q$-Weyl type.

Let $X \in \text{End}(V)$ and let $\{V_i\}_{i=0}^d$ be a decomposition of $V$.
We say $X$ is {\em tridiagonal with respect to $\{V_i\}_{i=0}^d$} whenever
\begin{align}
   X V_i \subseteq V_{i-1} + V_i + V_{i+1}  && (0 \leq i \leq d).    \label{eq:deftrid}
\end{align}
Let $\{E_i\}_{i=0}^d$ be the idempotent sequence for $\{V_i\}_{i=0}^d$.
Then $X$ satisfies \eqref{eq:deftrid} if and only if
$E_r X E_s = 0$ if $|r-s|>1$ $(0 \leq r,s \leq d)$.

Let $A,B,C$ be an LR triple on $V$.
Let ${\cal X}$ denote the subspace of $\text{End}(V)$ consisting of $X \in \text{End}(V)$
such that $X$ is tridiagonal with respect to the $(A,B)$ and $(B,C)$ and $(C,A)$ decompositions for $V$.
We call $\cal X$ the {\em tridiagonal space} for $A,B,C$.
Each LR triple on $V$ that is associated to $A,B,C$ has tridiagonal space $\cal X$
(see Corollary \ref{cor:invariant}).
The following elements are contained in ${\cal X}$ (see Lemma \ref{lem:contained1}):
\begin{equation}                  \label{eq:contained1}
 I, \quad
 A, \quad
 B, \quad
 C, \quad
 ABC, \quad
 BCA, \quad
 CAB, \quad
 ACB, \quad
 CBA, \quad
 BAC.
\end{equation} 
Moreover, if $A,B,C$ is bipartite, then $XJ \in {\cal X}$ for any $X \in {\cal X}$
(see Lemma \ref{lem:contained}).
In the present paper, we investigate the tridiagonal space ${\cal X}$.
Observe that ${\cal X} = \text{End}(V)$ when $d \leq 1$.
So we restrict our attention to the case $d \geq 2$.
We prove the following results:

\begin{theorem}            \label{thm:main2}    \samepage
\ifDRAFT {\rm thm:main2}. \fi
Let $A,B,C$ be a nonbipartite LR triple on $V$,
and let $\cal X$ be the corresponding tridiagonal space.
Assume that $A,B,C$ is not associated to an LR triple of $q$-Weyl type.
Then the following hold:
\begin{itemize}
\item[\rm (i)]
Assume $d=2$.
Then $\cal X$ has dimension $6$.
Moreover, the vector space $\cal X$ has a basis
\begin{equation}                                                  \label{eq:basisd2}
 I, \quad
 A, \quad
 B, \quad
 C, \;\quad
 A B C, \;\quad
 A C B.
\end{equation}
\item[\rm (ii)]
Assume $d \geq 3$.
Then $\cal X$ has dimension $7$.
Moreover, the vector space $\cal X$ has a basis
\begin{equation}                                                  \label{eq:basis}
 I, \quad
 A, \quad
 B, \quad
 C, \;\quad
 A B C, \;\quad
 A C B, \quad
 C A B.
\end{equation}
\end{itemize}
\end{theorem}

\begin{theorem}    \label{thm:main1}    \samepage
\ifDRAFT {\rm thm:main1}. \fi
Let $A,B,C$ be a bipartite LR triple on $V$,
and let ${\cal X}$ be the corresponding tridiagonal space.
Then ${\cal X}  = {\cal X} J + {\cal X} (I-J)$ (direct sum).
Moreover, the following hold:
\begin{itemize}
\item[\rm (i)]
Assume $d=2$.
Then ${\cal X}$ has dimension $6$.
The vector space ${\cal X} J$ has a basis
\begin{equation}                                                             \label{eq:basisd2XJ}
  J, \;\quad
  A J, \quad
  B J,
\end{equation}
and the space ${\cal X} (I-J)$ has a basis
\begin{equation}                                                       \label{eq:basisd2XJd}
  I-J, \quad
  A (I-J), \quad
  B (I-J).
\end{equation}
\item[\rm (ii)]
Assume $d \geq 4$.
Then $\cal X$ has dimension $8$.
The vector space  ${\cal X} J$ has a basis
\begin{equation}                                                             \label{eq:basisXJ}
  J, \;\quad
  A J, \quad
  B J, \quad
  A C B J,
\end{equation}
and the space ${\cal X} (I-J)$ has a basis
\begin{equation}                                                       \label{eq:basisXJd}
  I-J, \quad
  A (I-J), \quad
  B (I-J), \quad
  A B C (I-J).
\end{equation}
\end{itemize}
\end{theorem}

\begin{note}
When $A,B,C$ has $q$-Weyl type,
the elements
\[
  ABC, \quad
  BCA, \quad
  CAB, \quad
  ACB, \quad
  CBA, \quad
  BAC
\]
are contained in the span of $I$, $A$, $B$, $C$
(see \cite[Lemma 15.30]{T:LRT}).
\end{note}

\begin{problem}
For an LR triple $A,B,C$ of $q$-Weyl type,
find the dimension and a basis for the tridiagonal space $\cal X$.
\end{problem}

The paper is organized as follows.
In Section \ref{sec:bases} we consider 12 bases for $V$.
In Section \ref{sec:trans} we obtain
the transition matrices between these 12 bases.
In Section \ref{sec:idempotents} we obtain the matrices that represent
the idempotents with respect to the 12 bases.
In Section \ref{sec:lemmas} we prepare some lemmas concerning the tridiagonal space.
In Sections \ref{sec:nonbipartite}--\ref{sec:proofnonbipartite} we consider
nonbiparite LR triples.
In Section \ref{sec:nonbipartite} we recall the classification of nonbiparitite 
LR triples.
In Sections \ref{sec:boundnonbipartite} and \ref{sec:proofnonbipartite}
we prove Theorem \ref{thm:main2}.
In Sections \ref{sec:bipartite}--\ref{sec:proofbipartite} we consider
bipartite LR triples.
In Section \ref{sec:bipartite} we recall the classification of bipartite LR triples.
In Sections \ref{sec:boundbipartite} and \ref{sec:proofbipartite} we prove
Theorem \ref{thm:main1}.
In Appendix 1 we represent the elements \eqref{eq:contained1}
in terms of \eqref{eq:basis}.
In Appendix 2 we represent the elements \eqref{eq:contained1} times $J$
in terms of \eqref{eq:basisXJ},
and represent the elements \eqref{eq:contained1} times $I-J$
in terms of \eqref{eq:basisXJd}.

\section{Some bases for $V$}
\label{sec:bases}

Let $A,B$ be an LR pair on $V$, and let $\{V_i\}_{i=0}^d$ be the
$(A,B)$-decomposition of $V$.
By \cite[Lemma 3.12]{T:LRT},
for $0 \leq i \leq d$ the subspace $V_i$ is invariant under $AB$ and $BA$.
Moreover, for $1 \leq i \leq d$,
the eigenvalue of $AB$ on $V_{i-1}$ is nonzero and equal to the eigenvalue
of $BA$ on $V_i$. We denote this eigenvalue by $\vphi_i$.
The sequence $\{\vphi_i\}_{i=1}^d$ is called the  {\em parameter sequence}
for $A,B$.
We emphasize that $\vphi_i \neq 0$ for $1 \leq i \leq d$.
For notational convenience define $\vphi_0 = 0$ and $\vphi_{d+1}=0$.
A basis $\{v_i\}_{i=0}^d$ for $V$ is called an {\em $(A,B)$-basis} whenever
$v_i \in V_i$ for $0 \leq i \leq d$ and $Av_i = v_{i-1}$ for $1 \leq i \leq d$.
A basis $\{v_i\}_{i=0}^d$ for $V$ is called an {\em inverted $(A,B)$-basis}
whenever its inversion $\{v_{d-i}\}_{i=0}^d$ is an $(A,B)$-basis for $V$.

Let $A,B,C$ be an LR triple on $V$.
As we discuss this LR triple,
we will use the following notation:

\begin{definition}   {\rm (See \cite[Definition 13.4]{T:LRT}.) } \label{def:prime}  \samepage
\ifDRAFT {\rm def:prime}. \fi
Let $A,B,C$ be an LR triple.
For any object $f$ associated with the LR triple $A,B,C$,
let $f'$ (resp.\ $f''$) denote the corresponding object for the LR triple
$B,C,A$ (resp.\ $C,A,B$).
\end{definition}

\begin{definition}   {\rm (See \cite[Definition 13.21]{T:LRT}.) }   \samepage
Let $A,B,C$ be an LR triple on $V$.
So the pair $A,B$ (resp.\ $B,C$) (resp.\ $C,A$) is an LR pair on $V$.
Following the notational convention in Definition \ref{def:prime}, for these LR pairs
the parameter sequence is denoted as follows:
\[
  \begin{array}{c|c}
    \text{LR pair} & \text{parameter sequence}
  \\ \hline
    A, B &  \{\vphi_i\}_{i=1}^d
 \\
  B,C & \{\vphi'_i\}_{i=1}^d
 \\
  C,A & \{\vphi''_i\}_{i=1}^d
  \end{array}
\]
We call the sequence 
\begin{equation}
   (\{\vphi_i\}_{i=1}^d, \{\vphi'_i\}_{i=1}^d, \{\vphi''_i\}_{i=1}^d)    \label{eq:parray}
\end{equation}
the {\em parameter array} of the LR triple $A,B,C$.
\end{definition}

Let $A,B,C$ be an LR triple on $V$.
Associated with $A,B,C$ are $12$ types of bases for $V$:
\begin{equation}                               \label{eq:12bases}
\begin{array}{cccc}
 (A,B), \quad & \text{inverted $(A,B)$}, \quad &  (B,A), \quad & \text{inverted $(B,A)$},
\\
 (B,C), \quad & \text{inverted $(B,C)$}, \quad &  (C,B), \quad & \text{inverted $(C,B)$},
\\
 (C,A), \quad & \text{inverted $(C,A)$}, \quad &  (A,C), \quad & \text{inverted $(A,C)$}.
\end{array}
\end{equation}

\begin{lemma}  {\rm (See \cite[Lemma 13.19]{T:LRT}.) }  \label{lem:trid}    \samepage
\ifDRAFT {\rm lem:trid}. \fi
Let $\{v_i\}_{i=0}^d$ be a basis for $V$ that has one of the $12$ types \eqref{eq:12bases}.
Then for each $A,B,C$, the matrix representing it with respect to $\{v_i\}_{i=0}^d$
is tridiagonal.
\end{lemma}

Let \eqref{eq:parray} be the parameter array and let
\eqref{eq:tracedata} be the trace data of $A,B,C$.

\begin{lemma}    {\rm (See \cite[Proposition 13.39]{T:LRT}.) }
\label{lem:entriesABC}    \samepage
\ifDRAFT {\rm lem:entriesABC}. \fi
Fix an $(A,B)$-basis $\{v_i\}_{i=0}^d$ for $V$.
Identify each element of $\text{\rm End}(V)$ with the matrix representing it with
respect to $\{v_i\}_{i=0}^d$.
Then each of $A$, $B$, $C$ is tridiagonal with the following entries:
\[
 \begin{array}{ccc|ccc|ccc}
   A_{i,i-1} & A_{i,i} & A_{i-1,i} & 
   B_{i,i-1} & B_{i,i} & B_{i-1,i} & C_{i,i-1} & C_{i,i} & C_{i-1,i}
 \\ \hline
  0 & 0 & 1 & \vphi_i & 0 & 0 & \vphi''_{d-i+1} & a_i & \vphi'_{d-i+1}/\vphi_i   \rule{0mm}{4mm}
 \end{array}
\]
\end{lemma}

\section{Transition matrices}
\label{sec:trans}

Let $A,B,C$ be an LR triple on $V$ with parameter array \eqref{eq:parray}.
In this section, we consider the transition matrices between the 12 bases \eqref{eq:12bases}.

\begin{definition}   \label{def:compatible}   \samepage
\ifDRAFT {\rm def:compatible}. \fi
Two bases $\{v_i\}_{i=0}^d$ and $\{u_i\}_{i=0}^d$ for $V$ are said to be
{\em compatible} whenever $v_0 = u_0$.
\end{definition}

\begin{definition}  {\rm (See \cite[Definition 13.44]{T:LRT}.) }  \label{def:T}   \samepage
\ifDRAFT {\rm def:T}. \fi
Define matrices $T$, $T'$, $T''$ in $\Mat_{d+1}(\F)$ as follows:
\begin{itemize}
\item[]
$T$ is the transition matrix from a $(C,B)$-basis to a compatible $(C,A)$-basis;
\item[]
$T'$ is the transition matrix from an $(A,C)$-basis to a compatible $(A,B)$-basis;
\item[]
$T''$ is the transition matrix from a $(B,A)$-basis to a compatible $(B,C)$-basis.
\end{itemize}
\end{definition}

Let $\{\alpha_i\}_{i=0}^d$ be scalars in $\F$.
An upper triangular matrix $T \in \Mat_{d+1}(\F)$ is called {\em Toeplitz with parameters $\{\alpha_i\}_{i=0}^d$}
whenever $T$ has $(i,j)$-entry $\alpha_{j-i}$ for $0 \leq i \leq j \leq d$:
\[
 T = 
  \begin{pmatrix}
    \alpha_0 & \alpha_1 & \cdot & \cdot & \cdot & \alpha_d
  \\
    & \alpha_0 & \alpha_1 & \cdot & \cdot & \cdot
  \\
    & &  \alpha_0 & \cdot & \cdot & \cdot 
   \\
   & & &  \cdot & \cdot & \cdot 
   \\
   & & & & \cdot & \alpha_1
   \\
  \text{\bf 0} & & & & & \alpha_0
  \end{pmatrix}.
\]
This matrix is invertible if and only if $\alpha_0 \neq 0$.
In this case, $T^{-1}$ is upper triangular and Toeplitz (see \cite[Section 12]{T:LRT}).

\begin{lemma}  {\rm (See \cite[Proposition 12.8]{T:LRT}.) } \label{lem:Toeplitz}   \samepage
\ifDRAFT {\rm lem:Toeplitz}. \fi
With reference to Definition \ref{def:T}, each of $T$, $T'$, $T''$ is upper triangular and Toeplitz.
\end{lemma}

\begin{definition}  {\rm (See \cite[Definition 13.45]{T:LRT}.) }  \label{def:Toeplitzdata}   \samepage
\ifDRAFT {\rm def:Toeplitzdata}. \fi
With reference to Lemma \ref{lem:Toeplitz},
let $\{\alpha_i\}_{i=0}^d$ (resp.\ $\{\alpha'_i\}_{i=0}^d$)  (resp.\ $\{\alpha''_i\}_{i=0}^d$)
be the Toeplitz parameters for $T$ (resp.\ $T'$) (resp.\ $T''$).
Let $\{\beta_i\}_{i=0}^d$ (resp.\ $\{\beta'_i\}_{i=0}^d$) (resp.\ $\{\beta''_i\}_{i=0}^d$)
be the Toeplitz parameters for $T^{-1}$ (resp.\ $(T')^{-1}$) (resp.\ $(T'')^{-1}$).
We call the sequence
\begin{equation}
  (\{\alpha_i\}_{i=0}^d, \{\beta_i\}_{i=0}^d; 
   \{\alpha'_i\}_{i=0}^d, \{\beta'_i\}_{i=0}^d; 
   \{\alpha''_i\}_{i=0}^d, \{\beta''_i\}_{i=0}^d)        \label{eq:Toeplitzdata}
\end{equation}
the {\em Teoplitz data} for $A,B,C$.
\end{definition}

\begin{lemma}   {\rm (See \cite[Lemma 13.46]{T:LRT}.) }  \label{lem:13.46}    \samepage
\ifDRAFT {\rm lem:13.46}. \fi
With reference to Definition \ref{def:Toeplitzdata},
\begin{align*}
 \alpha_0 &= 1, &
 \alpha'_0 &= 1, &
 \alpha''_0 &= 1, &
 \beta_0 &= 1, &
 \beta'_0 &= 1, &
 \beta''_0 &= 1.
\end{align*}
Moreover, when $d \geq 1$,
\begin{align*}
 \beta_1 &= - \alpha_1, &
 \beta'_1 &= - \alpha'_1, &
 \beta''_1 &= - \alpha''_1.
\end{align*}
\end{lemma}

\begin{definition}   \label{def:matrixD}   \samepage
\ifDRAFT {\rm def:matrixD}. \fi
Let $Z$ denote the matrix in $\Mat_{d+1}(\F)$ that has $(i,j)$-entry
$\delta_{i+j,d}$ for $0 \leq i,j \leq d$.
For example if $d=3$,
\[
 Z = 
\begin{pmatrix}
 0 & 0 & 0 & 1  \\
 0 & 0 & 1 & 0  \\
 0 & 1 & 0 & 0  \\
 1 & 0 & 0 & 9
\end{pmatrix}.
\]
Observe that $Z$ is invertible and $Z^{-1} = Z$.
Let $D$ (resp.\ $D'$) (resp.\ $D''$)
denote the diagonal matrix  in $\Mat_{d+1}(\F)$ that has $(i,i)$-entry $\vphi_1 \cdots \vphi_i$
(resp.\ $\vphi'_1 \cdots \vphi'_i$) (resp.\ $\vphi''_1 \cdots \vphi''_i$) 
for $0 \leq i \leq d$.
\end{definition}

\begin{lemma} {\rm (See \cite[Lemma 3.48]{T:LRT}.) }   \label{lem:LRPtransition}   \samepage
\ifDRAFT {\rm lem:LRPtransition}. \fi
\begin{itemize}
\item[\rm (i)]
The transition matrix from an $(A,B)$-basis to an inverted $(A,B)$-basis is a nonzero scalar multiple
of $Z$.
\item[\rm (ii)]
The transition matrix from an $(A,B)$-basis to an inverted $(B,A)$-basis is a nonzero scalar multiple
of $D$.
\item[\rm (iii)]
The transition matrix from an $(A,B)$-basis to a $(B,A)$-basis is a nonzero scalar multiple 
of $DZ$.
\end{itemize}
\end{lemma}

\begin{lemma}    \label{lem:transAB}          \samepage
\ifDRAFT {\rm lem:transAB}. \fi
In the table below, the transition matrix from the basis in the first column to the basis in the second column
is a nonzero scalar multiple of the matrix in the third column:
\[
 \begin{array}{c|c|c}
\text{\rm from} & \text{\rm to} & \text{\rm transition matrix}
\\ 
\hline
 (A,B) & (A,B) & I  \rule{0mm}{2.5ex}
\\
 (A,B) & \text{\rm inv.$(A,B)$} & Z
\\
 (A,B) & (B,A) & DZ
\\
 (A,B) & \text{\rm inv.$(B,A)$} & D
\\ \hline
 (A,C) & (A,B) & T'    \rule{0mm}{2.5ex}
\\
 (A,C) & \text{\rm inv.$(A,B)$} &T' Z
\\
 (A,C) & (B,A) & T' D Z
\\
 (A,C) & \text{\rm inv.$(B,A)$} &T'D
\\ \hline
 (B,C) & (A,B) & (T'')^{-1} Z D^{-1}    \rule{0mm}{2.5ex}
\\
 (B,C) & \text{\rm inv.$(A,B)$} & (T'')^{-1} Z D^{-1} Z
\\
 (B,C) & (B,A) & (T'')^{-1}
\\
 (B,C) & \text{\rm inv.$(B,A)$} & (T'')^{-1} Z
\end{array}
\]
\end{lemma}

\begin{proof}
The transition matrices from an $(A,B)$-basis are given in Lemma \ref{lem:LRPtransition}.
By Definition \ref{def:T} the transition matrix from an $(A,C)$-basis to an $(A,B)$-basis is 
a nonzero scalar multiple of $T'$.
The remaining transition matrices from an $(A,C)$-basis are obtained by multiplying $T'$
on the right by the transition matrices from an $(A,B)$-basis.
By Definition \ref{def:T} the transition matrix from an $(B,A)$-basis to a $(B,C)$-basis is 
a nonzero scalar multiple of $T''$,
so the transition matrix from a $(B,C)$-basis to a $(B,A)$-basis is a nonzero scalar multiple of $(T'')^{-1}$.
The transition matrix from an $(A,B)$-basis to a $(B,A)$-basis is a nonzero scalar multiple of $DZ$, 
so the transition
matrix from a $(B,A)$-basis to an $(A,B)$-basis is a nonzero scalar multiple of $(DZ)^{-1}$.
By these comments, the transition matrix from a $(B,C)$-basis to an $(A,B)$-basis is
a nonzero scalar multiple of
$(T'')^{-1} (DZ)^{-1} = (T'')^{-1} Z D^{-1}$.
The remaining transition matrices from a $(B,C)$-basis are obtained by multiplying
$(T'')^{-1} Z D^{-1}$ on the right by the transition matrices from an $(A,B)$-basis.
\end{proof}

Applying Lemma \ref{lem:transAB} to the LR triple $B,C,A$ we obtain:

\begin{lemma}    \label{lem:transBC}          \samepage
\ifDRAFT {\rm lem:transBC}. \fi
In the table below, the transition matrix from the basis in the first column to the basis in the second column
is a nonzero scalar multiple of the matrix in the third column:
\[
 \begin{array}{c|c|c}
\text{\rm from} & \text{\rm to} & \text{\rm transition matrix}
\\ 
\hline
 (B,C) & (B,C) & I  \rule{0mm}{2.5ex}
\\
 (B,C) & \text{\rm inv.$(B,C)$} & Z
\\
 (B,C) & (C,B) & D' Z
\\
 (B,C) & \text{\rm inv.$(C,B)$} & D'
\\ \hline
 (B,A) & (B,C) & T''    \rule{0mm}{2.5ex}
\\
 (B,A) & \text{\rm inv.$(B,C)$} &T'' Z
\\
 (B,A) & (C,B) & T'' D' Z
\\
 (B,A) & \text{\rm inv.$(C,B)$} &T'' D'
\\ \hline
 (C,A) & (B,C) & T^{-1} Z (D')^{-1}    \rule{0mm}{2.5ex}
\\
 (C,A) & \text{\rm inv.$(B,C)$} & T^{-1} Z (D')^{-1} Z
\\
 (C,A) & (C,B) & T^{-1}
\\
 (C,A) & \text{\rm inv.$(C,B)$} & T^{-1} Z
\end{array}
\]
\end{lemma}

Applying Lemma \ref{lem:transAB} to the LR triple $C,A,B$ we obtain:

\begin{lemma}    \label{lem:transCA}          \samepage
\ifDRAFT {\rm lem:transCA}. \fi
In the table below, the transition matrix from the basis in the first column to the basis in the second column
is a nonzero scalar multiple of the matrix in the third column:
\[
 \begin{array}{c|c|c}
\text{\rm from} & \text{\rm to} & \text{\rm transition matrix}
\\ 
\hline
 (C,A) & (C,A) & I  \rule{0mm}{2.5ex}
\\
 (C,A) & \text{\rm inv.$(C,A)$} & Z
\\
 (C,A) & (A,C) & D'' Z
\\
 (C,A) & \text{\rm inv.$(A,C)$} & D''
\\ \hline
 (C,B) & (C,A) & T    \rule{0mm}{2.5ex}
\\
 (C,B) & \text{\rm inv.$(C,A)$} &T Z
\\
 (C,B) & (A,C) & T D'' Z
\\
 (C,B) & \text{\rm inv.$(A,C)$} &T D''
\\ \hline
 (A,B) & (C,A) & (T')^{-1} Z (D'')^{-1}    \rule{0mm}{2.5ex}
\\
 (A,B) & \text{\rm inv.$(C,A)$} & (T')^{-1} Z (D'')^{-1} Z
\\
 (A,B) & (A,C) & (T')^{-1}
\\
 (A,B) & \text{\rm inv.$(A,C)$} & (T')^{-1} Z
\end{array}
\]
\end{lemma}

\section{Representing the idempotents with respect to the 12 bases}
\label{sec:idempotents}

In this section, we obtain the matrices that represent the idempotents
with respect to the 12 bases \eqref{eq:12bases}.
We begin by recalling a lemma from elementary linear algebra:

\begin{lemma}   \label{lem:trans}   \samepage
\ifDRAFT {\rm lem:trans}. \fi
Let $H \in \text{\rm End}(V)$, and
$\{u_i\}_{i=0}^d$, $\{v_i\}_{i=0}^d$ be bases for $V$.
Let $M$ be the matrix representing $H$ with respect to $\{u_i\}_{i=0}^d$,
and let  $S$ be the transition matrix from $\{u_i\}_{i=0}^d$ to $\{v_i\}_{i=0}^d$.
Then the matrix representing $H$ with respect to $\{v_i\}_{i=0}^d$ is
$S^{-1} M S$.
\end{lemma}

We use the following notation:

\begin{definition}   \label{def:Fr}   \samepage
\ifDRAFT {\rm def:Fr}. \fi
For $0 \leq r \leq d$ let $F_r$ denote the matrix in $\text{Mat}_{d+1}(\F)$
that has $(r,r)$-entry $1$ and all other entries $0$.
\end{definition}

Let $A,B,C$ be an LR triple on $V$ with parameter array \eqref{eq:parray},
Toeplitz data \eqref{eq:Toeplitzdata},
and idempotent data \eqref{eq:idempotentdata}.
The following three propositions are routinely obtained 
from Lemmas \ref{lem:transAB}--\ref{lem:transCA} and \ref{lem:trans}.

\begin{proposition}   \label{prop:Er}    \samepage
\ifDRAFT {\rm prop:Er}. \fi
For $0 \leq r \leq d$,
with respect to the basis in the first column,  $E_r$ is represented by the matrix
in the second column:
\[
\begin{array}{c|c}
\text{\rm basis} & \text{\rm the matrix representing $E_r$}
\\ \hline
(A,B) & F_r      \rule{0mm}{2.5ex}
\\
\text{\rm inv.$(A,B)$} & F_{d-r}
\\
(B,A) &  F_{d-r}
\\
\text{\rm inv.$(B,A)$} &  F_r
\\ \hline
(B,C) & (T'')^{-1} F_{d-r} T''      \rule{0mm}{2.5ex}
\\
\text{\rm inv.$(B,C)$} & Z (T'')^{-1} F_{d-r} T'' Z
\\
(C,B) & Z (D')^{-1} (T'')^{-1} F_{d-r} T'' D' Z
\\
\text{\rm inv.$(C,B)$} & (D')^{-1} (T'')^{-1} F_{d-r} T'' D'
\\ \hline
(C,A) & D'' Z T' F_r (T')^{-1} Z (D'')^{-1}      \rule{0mm}{2.5ex}
\\
\text{\rm inv.$(C,A)$} & Z D'' Z T' F_r (T')^{-1} Z (D'')^{-1} Z
\\
(A,C) & T' F_r (T')^{-1}
\\
\text{\rm inv.$(A,C)$} & Z T' F_r (T')^{-1} Z
\end{array}
\]
\end{proposition}

\begin{proposition}   \label{prop:Edr}   \samepage
\ifDRAFT {\rm prop:Edr}. \fi
For $0 \leq r \leq d$,
with respect to the basis in the first column,  $E'_r$ is represented by the matrix
in the second column:
\[
\begin{array}{c|c}
\text{\rm basis} & \text{\rm the matrix representing $E'_r$}
\\ \hline
(A,B) & D Z T'' F_r (T'')^{-1} Z D^{-1}      \rule{0mm}{2.5ex}
\\
\text{\rm inv.$(A,B)$} & Z D Z T'' F_r (T'')^{-1} Z D^{-1} Z
\\
(B,A) & T'' F_r (T'')^{-1}
\\
\text{\rm inv.$(B,A)$} & Z T'' F_r (T'')^{-1} Z
\\ \hline
(B,C) & F_r     \rule{0mm}{2.5ex}
\\
\text{\rm inv.$(B,C)$} & F_{d-r}
\\
(C,B) & F_{d-r}
\\
\text{\rm inv.$(C,B)$} & F_r
\\ \hline
(C,A) & T^{-1} F_{d-r} T     \rule{0mm}{2.5ex}
\\
\text{\rm inv.$(C,A)$} & Z T^{-1} F_{d-r} T Z
\\
(A,C) &  Z (D'')^{-1} T^{-1} F_{d-r} T D'' Z
\\
\text{\rm inv.$(A,C)$} &(D'')^{-1} T^{-1} F_{d-r} T D''
\end{array}
\]
\end{proposition}

\begin{proposition}   \label{prop:Eddr}   \samepage
\ifDRAFT {\rm prop:Eddr}. \fi
For $0 \leq r \leq d$,
with respect to the basis in the first column,  $E''_r$ is represented by the matrix
in the second column:
\[
\begin{array}{c|c}
\text{\rm basis} & \text{\rm the matrix representing $E''_r$}
\\ \hline
(A,B) & (T')^{-1} F_{d-r} T'      \rule{0mm}{2.5ex}
\\
\text{\rm inv.$(A,B)$} & Z (T')^{-1} F_{d-r} T' Z
\\
(B,A) &  Z D^{-1} (T')^{-1} F_{d-r} T' D Z
\\
\text{\rm inv.$(B,A)$} & D^{-1} (T')^{-1} F_{d-r} T' D
\\ \hline
(B,C) & D' Z T F_r T^{-1} Z (D')^{-1}    \rule{0mm}{2.5ex}
\\
\text{\rm inv.$(B,C)$} & Z D' Z T F_r T^{-1} Z (D')^{-1} Z
\\
(C,B) &  T F_r T^{-1}
\\
\text{\rm inv.$(C,B)$} & Z T F_r T^{-1} Z
\\ \hline
(C,A) & F_r   \rule{0mm}{2.5ex}
\\
\text{\rm inv.$(C,A)$} & F_{d-r}
\\
(A,C) &  F_{d-r}
\\
\text{\rm inv.$(A,C)$} & F_r
\end{array}
\]
\end{proposition}

The following three proposition are obtained by computing the entries 
of the matrices given in Propositions \ref{prop:Er}--\ref{prop:Eddr}.

\begin{proposition}   \label{prop:Erij}    \samepage
\ifDRAFT {\rm prop:Erij}. \fi
For $0 \leq r,i,j \leq d$,
with respect to the basis in the first column,  the matrix representing $E_r$ 
has $(i,j)$-entry in the second column when the condition in the third column is satisfied,
and $0$ otherwise:
\[
\begin{array}{c|c|c}
\text{\rm basis} & \text{\rm $(i,j)$-entry of the matrix representing $E_r$} & \text{\rm condition}
\\ \hline
(A,B) & 1 & i=r=j     \rule{0mm}{2.5ex}
\\
\text{\rm inv.$(A,B)$} & 1 & i=d-r=j 
\\
(B,A) & 1 & i=d-r=j
\\
\text{\rm inv.$(B,A)$} & 1 & i=r=j
\\ \hline
(B,C) & \alpha''_{r-d+j} \beta''_{d-r-i} & i \leq d-r \leq j      \rule{0mm}{2.5ex}
\\
\text{\rm inv.$(B,C)$} & \alpha''_{r-j} \beta''_{i-r} & j \leq r \leq i
\\
(C,B) & \alpha''_{r-j} \beta''_{i-r} \vphi'_{d-i+1} \cdots \vphi_{d-j} & j \leq r \leq i
\\
\text{\rm inv.$(C,B)$} & \alpha''_{r-d+j} \beta''_{d-r-i} \vphi'_{i+1} \cdots \vphi'_j & i \leq d-r \leq j
\\ \hline
(C,A) & \alpha'_{r-d+i} \beta'_{d-r-j} \vphi''_{j+1} \cdots \vphi''_i & j \leq d-r \leq i     \rule{0mm}{2.5ex}
\\
\text{\rm inv.$(C,A)$} & \alpha'_{r-i} \beta'_{j-r} \vphi''_{d-j+1} \cdots \vphi''_{d-i} & i \leq r \leq j
\\
(A,C) & \alpha'_{r-i} \beta'_{j-r} & i \leq r \leq j
\\
\text{\rm inv.$(A,C)$} & \alpha'_{r-d+i} \beta'_{d-r-j}  & j \leq d-r \leq i
\end{array}
\]
\end{proposition}

\begin{proposition}   \label{prop:Edrij}    \samepage
\ifDRAFT {\rm prop:Edrij}. \fi
For $0 \leq r,i,j \leq d$,
with respect to the basis in the first column,  the matrix representing $E'_r$ 
has $(i,j)$-entry in the second column when the condition in the third column is satisfied,
and $0$ otherwise:
\[
\begin{array}{c|c|c}
\text{\rm basis} & \text{\rm $(i,j)$-entry of the matrix representing $E'_r$} & \text{\rm condition}
\\ \hline
(A,B) & \alpha''_{r-d+i} \beta''_{d-r-j} \vphi_{j+1} \cdots \vphi_i& j \leq d-r \leq i     \rule{0mm}{2.5ex}
\\
\text{\rm inv.$(A,B)$} & \alpha''_{r-i} \beta''_{j-r} \vphi_{d-j+1} \cdots \vphi_{d-i} & i \leq r \leq j
\\
(B,A) & \alpha''_{r-i} \beta''_{j-r} & i \leq r \leq j
\\
\text{\rm inv.$(B,A)$} & \alpha''_{r-d+i} \beta''_{d-r-j} & j \leq d-r \leq i
\\ \hline
(B,C) & 1 & i = r = j      \rule{0mm}{2.5ex}
\\
\text{\rm inv.$(B,C)$} & 1 & i = d-r = j
\\
(C,B) & 1 & i = d-r = j
\\
\text{\rm inv.$(C,B)$} & 1 & i = r = j
\\ \hline
(C,A) & \alpha_{r-d+j} \beta_{d-r-i} & i \leq d-r \leq j     \rule{0mm}{2.5ex}
\\
\text{\rm inv.$(C,A)$} & \alpha_{r-j} \beta_{i-r}  & j \leq r \leq i
\\
(A,C) & \alpha_{r-j} \beta_{i-r} \vphi''_{d-i+1} \cdots \vphi''_{d-j} & j \leq r \leq i
\\
\text{\rm inv.$(A,C)$} & \alpha_{r-d+j} \beta_{d-r-i} \vphi''_{i+1} \cdots \vphi''_j & i \leq d-r \leq j
\end{array}
\]
\end{proposition}

\begin{proposition}   \label{prop:Eddrij}    \samepage
\ifDRAFT {\rm prop:Eddrij}. \fi
For $0 \leq r,i,j \leq d$,
with respect to the basis in the first column,  the matrix representing $E''_r$ 
has $(i,j)$-entry in the second column when the condition in the third column is satisfied,
and $0$ otherwise:
\[
\begin{array}{c|c|c}
\text{\rm basis} & \text{\rm $(i,j)$-entry of the matrix representing $E''_r$} & \text{\rm condition}
\\ \hline
(A,B) &  \alpha'_{r-d+j} \beta'_{d-r-i} & i \leq d-r \leq j  \rule{0mm}{2.5ex}
\\
\text{\rm inv.$(A,B)$} &  \alpha'_{r-j} \beta'_{i-r} & j \leq r \leq i
\\
(B,A) & \alpha'_{r-j} \beta'_{i-r} \vphi_{d-i+1} \cdots \vphi_{d-j} & j \leq r \leq i
\\
\text{\rm inv.$(B,A)$} & \alpha'_{r-d+j} \beta'_{d-r-i} \vphi_{i+1} \cdots \vphi_j & i \leq d-r \leq j
\\ \hline
(B,C) &   \alpha_{r-d+i} \beta_{d-r-j} \vphi'_{j+1} \cdots \vphi'_i & j \leq d-r \leq i        \rule{0mm}{2.5ex}
\\
\text{\rm inv.$(B,C)$} & \alpha_{r-i} \beta_{j-r} \vphi'_{d-j+1} \cdots \vphi'_{d-i} & i \leq r \leq j
\\
(C,B) & \alpha_{r-i} \beta_{j-r} & i \leq r \leq j
\\
\text{\rm inv.$(C,B)$} &   \alpha_{r-d+i} \beta_{d-r-j} & j \leq d-r \leq i
\\ \hline
(C,A) &  1 & i=r=j   \rule{0mm}{2.5ex}
\\
\text{\rm inv.$(C,A)$} & 1 & i = d-r = j
\\
(A,C) &  1 & i = d-r = j
\\
\text{\rm inv.$(A,C)$} & 1 & i = r = j
\end{array}
\]
\end{proposition}

\section{Some lemmas concerning the tridiagonal space}
\label{sec:lemmas}

In this section we prepare some lemmas concerning the tridiagonal space
that we need in our proof of Theorems \ref{thm:main2} and \ref{thm:main1}.
Let $A,B,C$ be an LR triple on $V$.

\begin{lemma}  {\rm (See \cite[Lemma 13.22]{T:LRT}.) }
\label{lem:nonbipartitescalar}   \samepage
\ifDRAFT {\rm lem:nonbipartitescalar}. \fi
Let $\alpha$, $\beta$, $\gamma$ be nonzero scalars in $\F$.
Then the $3$-tuple $\alpha A, \beta B, \gamma C$ is an LR triple on $V$.
Moreover, the idempotent data of this LR triple is equal to the idempotent data
of $A,B,C$.
\end{lemma}

\begin{corollary}   \label{cor:invariant}    \samepage
\ifDRAFT {\rm cor:invariant}. \fi
With reference to Lemma \ref{lem:nonbipartitescalar},
the LR triples $A,B,C$ and $\alpha A, \beta B, \gamma C$ have the
same tridiagonal space.
\end{corollary}

\begin{proof}
Follows from Lemma \ref{lem:nonbipartitescalar}.
\end{proof}

\begin{lemma}  {\rm (See \cite[Lemma 13.19]{T:LRT}.)}  
\label{lem:action}  \samepage
\ifDRAFT {\rm lem:action}. \fi
In each row of the table below,
we display a decomposition $\{V_i\}_{i=0}^d$ of $V$.
For $0 \leq i \leq d$ we give the action of $A$, $B$, $C$ on $V_i$.
\[
\begin{array}{c|ccc}
\text{\rm dec. $\{V_i\}_{i=0}^d$} 
& \text{\rm action of $A$ on $V_i$} 
& \text{\rm action of $B$ on $V_i$} 
& \text{\rm action of $C$ on $V_i$} 
\\ \hline
(A,B) &                     \rule{0mm}{5mm}
A V_i = V_{i-1} &
B V_i = V_{i+1} &
C V_i \subseteq V_{i-1} + V_i + V_{i+1}
\\
(B,C) &                     \rule{0mm}{4mm}
A V_i \subseteq V_{i-1} + V_i + V_{i+1} &
B V_i = V_{i-1} &
C V_i = V_{i+1}
\\
(C,A) &               \rule{0mm}{4mm}
A V_i = V_{i+1} &
B V_i \subseteq V_{i-1} + V_i + V_{i+1} &
C V_i = V_{i-1}
\end{array}
\]
\end{lemma}

\begin{lemma}    \label{lem:contained1}    \samepage
\ifDRAFT {\rm lem:contained1}. \fi
The elements \eqref{eq:contained1} are contained in the tridiagonal space 
for $A,B,C$.
\end{lemma}

\begin{proof}
Let $\cal X$ be the tridiagonal space for $A,B,C$.
Clearly $I$ is contained in $\cal X$.
By Lemma \ref{lem:action} each of $A$, $B$, $C$ is tridiagonal with respect to
the $(A,B)$-basis, $(B,C)$-basis, and $(C,A)$-basis.
So $A,B,C$ are contained in $\cal X$.
We show $ABC$ is contained in $\cal X$.
Let $\{V_i\}_{i=0}^d$ be the $(A,B)$-decomposition of $V$.
Pick any $i$ $(0 \leq i \leq d)$.
By Lemma \ref{lem:action} $A B V_i \subseteq V_i$,
and $C V_i \subseteq V_{i-1} + V_i + V_{i+1}$.
By these comments $ABC V_i \subseteq V_{i-1} + V_i + V_{i+1}$.
So $ABC$ is tridiagonal with respect to $\{V_i\}_{i=0}^d$.
In a similar way, we can show that $ABC$ is tridiagonal with respect to
the $(B,C)$-decomposition and the $(C,A)$-decomposition.
We have shown that $ABC$ is contained in $\cal X$.
The proof is similar for the remaining elements in \eqref{eq:contained1}.
\end{proof}

\begin{lemma}    \label{lem:dimXd2}    \samepage
\ifDRAFT {\rm lem:dimXd2}. \fi
Assume $d=2$.
Then the tridiagonal space for $A,B,C$ has dimension at most $6$.
\end{lemma}

\begin{proof}
Let $\cal X$ be the tridiagonal space for $A,B,C$.
Fix an $(A,B)$-basis $\{v_i\}_{i=0}^d$ for $V$.
We identify each element of $\text{End}(V)$ with the matrix
that represents it with respect to $\{v_i\}_{i=0}^d$.
Pick any $X \in {\cal X}$.
By construction,
$X$ is tridiagonal, so $X_{0,2}=0$ and $X_{2,0}=0$.
By the definition of $\cal X$ we have
$E''_2 X E''_0 = 0$.
Let \eqref{eq:Toeplitzdata} be the Toeplitz data of $A,B,C$.
By Proposition \ref{prop:Eddrij}
\begin{align*}
  (E''_2)_{0,i} &= \alpha'_i \beta'_0 && (0 \leq i \leq d),
\\
 (E''_0)_{j,2} &= \alpha'_0 \beta'_{2-j}   && (0 \leq j \leq d).
\end{align*}
Using this we compute the $(0,2)$-entry of $E''_2 X E''_0$ to find
that $\alpha'_0 \beta'_0$ times
\begin{equation}     \label{eq:zero}
  \alpha'_0 \beta'_2 X_{0,0} + \alpha'_0 \beta'_1 X_{0,1} 
 + \alpha'_1 \beta'_2 X_{1,0} + \alpha'_1 \beta'_1 X_{1,1} + \alpha'_1 \beta'_0 X_{1,2}
 + \alpha'_2 \beta'_1 X_{2,1} + \alpha'_2 \beta'_0 X_{2,2}
\end{equation}
is $0$.
By Lemma \ref{lem:13.46} $\alpha'_0=1$ and $\beta'_0=1$.
So \eqref{eq:zero} is $0$.
By \cite[Lemma 13.62]{T:LRT} $\beta'_2 \neq 0$.
Therefore $X_{0,0}$ is uniquely determined by the $6$ entries
$X_{0,1}$, $X_{1,0}$, $X_{1,1}$, $X_{1,2}$, $X_{2,1}$, $X_{2,2}$.
Therefore $\cal X$ has dimension at most $6$.
\end{proof}

\section{Nonbipartite LR triples}
\label{sec:nonbipartite}

In this section, we recall the classification of nonbipartite LR triples.
Let $A,B,C$ be an LR triple on $V$ with parameter array \eqref{eq:parray},
Toeplitz data \eqref{eq:Toeplitzdata},
and idempotent data \eqref{eq:idempotentdata}.
To avoid triviality, we assume $d \geq 1$.

\begin{lemma}   {\rm (See \cite[Lemma 16.5]{T:LRT}.) }  \label{lem:alpha1}   \samepage
\ifDRAFT {\rm lem:alpha1}. \fi
Assume $A,B,C$ is nonbipartite.
Then each of
$\alpha_1$, $\alpha'_1$, $\alpha''_1$, $\beta_1$, $\beta'_1$, $\beta''_1$ is nonzero.
\end{lemma}

\begin{definition}   {\rm (See \cite[Definitions 18.2]{T:LRT}.) }
\label{def:normalized1}   \samepage
\ifDRAFT {\rm def:normalized1}. \fi
Assume $A,B,C$ is nonbipartite.
Then $A,B,C$ is said to be {\em normalized}
whenever $\alpha_1=1$, $\alpha'_1=1$, $\alpha''=1$.
\end{definition}

\begin{lemma}   {\rm (See \cite[Corollary 18.6]{T:LRT}.) }
\label{lem:normalized1}   \samepage
\ifDRAFT {\rm lem:normalized1}. \fi
Assume $A,B,C$ is nonbipartite.
Then there exist nonzero scalars $\alpha$, $\beta$, $\gamma$ in $\F$
such that $\alpha A, \beta B, \gamma C$ is normalized.
\end{lemma}

\begin{lemma}  {\rm (See \cite[Lemmas 17.6, 18.3]{T:LRT}.) }
\label{lem:nonbipartitenormalized} \samepage
\ifDRAFT {\rm lem:nonbipartitenormalized}. \fi
Assume $A,B,C$ is nonbipartite and normalized.
Then $\alpha_i = \alpha'_i = \alpha''_i$ for $0 \leq i \leq d$.
Moreover $\vphi_i = \vphi'_i = \vphi''_i$ for $1 \leq i \leq d$.
\end{lemma}

Let $A,B,C$ be a normalized LR triple on $V$ that is not $q$-Weyl type.
Assume $d \geq 2$,
and let \eqref{eq:parray} be the parameter array of $A,B,C$.
By the classification in \cite[Sections 26, 28, 29]{T:LRT},
$A,B,C$ is isomorphic to one of the following three types of LR triples:

\begin{definition}  {\rm (See \cite[Example 28.1]{T:LRT}.) }  \label{def:NBGq}    \samepage
\ifDRAFT {\rm lem:NBGq}. \fi
The LR triple $\text{NBG}_d(\F;q)$ is over $\F$, diameter $d$, nonbipartite, normalized,
and satisfies
\begin{align*}
 d &\geq 2;  \qquad\qquad 0 \neq q \in \F;
\\
  q &\neq 1 \quad (1 \leq i \leq d);  \qquad\qquad q^{d+1} \neq -1;
\\
 \vphi_i &= \frac{q(q^i-1)(q^{i-d-1}-1)}
                      {(q-1)^2}                    \qquad\qquad (1 \leq i \leq d).
\end{align*}
\end{definition}

\begin{definition}     {\rm (See \cite[Example 28.2]{T:LRT}.) }    \label{def:NBG1}    \samepage
\ifDRAFT {\rm lem:NBG1}. \fi
The LR triple $\text{NBG}_d(\F;1)$ is over $\F$, diameter $d$, nonbipartite, normalized,
and satisfies
\begin{align*}
 d &\geq 2;  \qquad\qquad \text{ $\text{\rm Char}(\F)$ is $0$ or greater than $d$};
\\
 \vphi_i &= i(i-d-1)            \qquad\qquad (1 \leq i \leq d).
\end{align*}
\end{definition}

\begin{definition}     {\rm (See \cite[Example 29.1]{T:LRT}.) } \label{def:NBNG}    \samepage
\ifDRAFT {\rm lem:NBNG}. \fi
The LR triple $\text{NBNG}_d(\F;t)$ is over $\F$, diameter $d$, nonbipartite, normalized,
and satisfies
\begin{align*}
 d &\geq 4;  \qquad\qquad \text{$d$ is even};  \qquad \qquad 0 \neq t \in \F;
\\
  t^i &\neq 1 \quad (1 \leq i \leq d/2);  \qquad\qquad t^{d+1} \neq 1;
\\
 \vphi_i &= 
  \begin{cases}
    t^{i/2} -1 & \text{ if $i$ is even},
  \\
  t^{(i-d-1)/2} -1 & \text{ if $i$ is odd}
 \end{cases}  
            \qquad\qquad (1 \leq i \leq d).
\end{align*}
\end{definition}

\section{Bounding the dimension of the tridiagonal space; nonbipartite case}
\label{sec:boundnonbipartite}

Let $A,B,C$ be an LR triple on $V$ with parameter array \eqref{eq:parray},
Toeplitz data \eqref{eq:Toeplitzdata},
and idempotent data \eqref{eq:idempotentdata}.
For $a$, $q \in \F$ and an integer $n \geq 0$, define
\[
  (a;q)_n = (1-a)(1-a q)\cdots (1-a q^{n-1}).
\]
We interpret $(a;q)_0=1$.

\begin{lemma} {\rm (See \cite[Proposition 31.3]{T:LRT})}   \label{lem:alpha}  \samepage
\ifDRAFT {\rm lem:alpha}. \fi
The following hold:
\begin{itemize}
\item[\rm (i)]
For the LR triple $\text{\rm NBG}_d(\F;q)$,
\begin{align*}
  \alpha_i &= \frac{(1-q)^i}
                         {(q;q)_i}          &&  (0 \leq i \leq d).
\end{align*}
\item[\rm (ii)]
For the LR triple  $\text{\rm NBG}_d(\F;1)$,
\begin{align*}
  \alpha_i &= \frac{1}
                         {i !}          &&  (0 \leq i \leq d).
\end{align*}
\item[\rm (iii)]
For the LR triple $\text{\rm NBNG}_d(\F;t)$,
\begin{align*}
 \alpha_i &= 
  \begin{cases}
    \frac{ 1 }
           { (t;t)_{i/2} }                    & \text{ if $i$ is even},
   \\
   \frac{1 }
          { (t;t)_{(i-1)/2} }                  & \text{ if $i$ is odd}
  \end{cases}
                          && (0 \leq i \leq d).
\end{align*}
\end{itemize}
\end{lemma}

Assume $A,B,C$ is one of 
$\text{\rm NBG}_d(\F;q)$, $\text{\rm NBG}_d(\F;1)$, $\text{\rm NBNG}_d(\F;t)$.
Let ${\cal X}$ be the tridiagonal space for $A,B,C$.

\begin{lemma}   \label{lem:cond}  \samepage
\ifDRAFT {\rm lem:cond}. \fi
The following hold:
\begin{itemize}
\item[\rm (i)]
$\alpha_i \neq 0$ for $0 \leq i \leq d$.
\item[\rm (ii)]
$\alpha_i^2 \vphi_i \neq \alpha_{i-1} \alpha_{i+1} \vphi_{i+1}$ for  $1 \leq i \leq d-1$.
\end{itemize}
\end{lemma}

\begin{proof}
(i):
Follows from Lemma \ref{lem:alpha}.

(ii):
We show
\begin{align}
    \frac{\alpha_i \vphi_i}
         {\alpha_{i-1} } 
 - \frac{\alpha_{i+1} \vphi_{i+1} }
          {\alpha_i} 
   & \neq 0   && (1 \leq i \leq d-1).        \label{eq:cond}
\end{align}
First consider $\text{\rm NBG}_d(\F;q)$.
Using Definition \ref{def:NBGq} and Lemma \ref{lem:alpha}(i), one checks
\begin{align*}
  \frac{\alpha_i \vphi_i}
         {\alpha_{i-1} } 
 - \frac{\alpha_{i+1} \vphi_{i+1} }
          {\alpha_i}
 &= - q^{i-d}
  &&   (1 \leq i \leq d-1).
\end{align*}
So \eqref{eq:cond} holds.
Next consider $\text{NBG}_d(\F;1)$.
Using Definition \ref{def:NBG1} and Lemma \ref{lem:alpha}(ii), one checks
\begin{align*}
  \frac{\alpha_i \vphi_i}
         {\alpha_{i-1} } 
 - \frac{\alpha_{i+1} \vphi_{i+1} }
          {\alpha_i}
 &= - 1
  &&   (1 \leq i \leq d-1).
\end{align*}
So \eqref{eq:cond} holds.
Next consider $\text{\rm NBNG}_d(\F;t)$.
Using Definition \ref{def:NBNG} and Lemma \ref{lem:alpha}(iii), one checks
\begin{align*}
  \frac{\alpha_i \vphi_i}
         {\alpha_{i-1} } 
 - \frac{\alpha_{i+1} \vphi_{i+1} }
          {\alpha_i}
 &= 
  \begin{cases}
    - t^{(i-d)/2} & \text{ if $i$ is even},
   \\
   t^{(i-d-1)/2} & \text{ if $i$ is odd}
  \end{cases}
  &&   (1 \leq i \leq d-1).
\end{align*}
So \eqref{eq:cond} holds.
The result follows.
\end{proof}

\begin{lemma}    \label{lem:vanish}   \samepage
\ifDRAFT {\rm lem:vanish}. \fi
Let $X \in {\cal X}$ such that
\begin{align*}
  X E'_d &= 0,  &
  X A E'_d &= 0, &
  E''_d X &= 0.
\end{align*}
Then $X=0$.
\end{lemma}

\begin{proof}
Note that $\alpha_i \neq 0$ for $0 \leq i \leq d$ by Lemma \ref{lem:cond}(i).
Fix an $(A,B)$-basis $\{v_i\}_{i=0}^d$ for $V$.
We identify each element of $\text{End}(V)$ with the matrix that represent it with respect to $\{v_i\}_{i=0}^d$.
By Propositions \ref{prop:Edrij}, \ref{prop:Eddrij}, and Lemmas \ref{lem:13.46}, \ref{lem:nonbipartitenormalized},
\begin{align*}
 (E'_d)_{i,0} &= \alpha_i \vphi_1 \cdots \vphi_i     && (0 \leq i \leq d),  
\\
 (E''_d)_{0,j} &=   \alpha_j     && (0 \leq j \leq d). 
\end{align*}
Observe that $X$ is tridiagonal since $X \in {\cal X}$.
For $0 \leq i \leq d$ let $x_i$ be the $(i,i)$-entry of $X$.
For $1 \leq i \leq d$ let $y_{i-1}$ be the $(i-1,i)$-entry of $X$,
and let $z_i$ be the $(i,i-1)$-entry of $X$:
\[
 X = 
 \begin{pmatrix}
  x_0 & y_0 &  & & & \text{\bf 0}  \\
  z_1 & x_1 & y_1 \\
       & z_2 & \cdot & \cdot  \\
       &      & \cdot & \cdot & \cdot  \\
       &      &          & \cdot & \cdot & y_{d-1} \\
  \text{\bf 0}   &      &          &          & z_d & x_d
 \end{pmatrix}.
\]
For $0 \leq i \leq d$, compute the $(i,0)$-entry of $X E'_d$ to find
\begin{align}
  0 &= \alpha_0 x_0 + \alpha_1 \vphi_1 y_0,                   \label{eq:10}
\\
  0 &= \alpha_i x_i + \alpha_{i+1} \vphi_{i+1} y_i + \alpha_{i-1} \vphi_i^{-1} z_i
                     && (1 \leq i \leq d-1),                          \label{eq:1i}
\\
 0 &= \alpha_d x_d + \alpha_{d-1} \vphi_d^{-1} z_d.          \label{eq:1d}
\end{align}
By Lemma \ref{lem:entriesABC}, for $0 \leq i,j \leq d$,
$A_{i,j}=1$ if $j=i+1$ and $A_{i,j}=0$ if $j \neq i+1$.
Using this, for $0 \leq i \leq d$, compute the $(i,0)$-entry of $X A E'_d$ to find
\begin{align}
  0 &= \alpha_1 x_0 + \alpha_2 \vphi_2 y_0,                    \label{eq:20}
\\
  0 &= \alpha_{i+1} x_i + \alpha_{i+2} \vphi_{i+2} y_i + \alpha_i \vphi_{i+1}^{-1} z_i 
                 && (1 \leq i \leq d-2),                               \label{eq:2i}
\\
 0 &= \alpha_d x_{d-1} + \alpha_{d-1} \vphi_d^{-1} z_{d-1},    \label{eq:2i-1}
\\
 0 &= \alpha_d z_d.                                                       \label{eq:2d}
\end{align}
For $0 \leq i \leq d-2$,
compute the $(0,i)$-entry of $E''_d X$ to find
\begin{align}
 0 &= \alpha_0 x_0 + \alpha_1 z_1,                                  \label{eq:30}
\\
 0 &= \alpha_i x_i + \alpha_{i-1} y_{i-1} + \alpha_{i+1} z_{i+1} 
         && (1 \leq i \leq d-2).                                        \label{eq:3i}
\end{align}
In \eqref{eq:1i} for $i=1$ and \eqref{eq:30}, eliminate $z_1$ to get
\begin{equation}
 \alpha_1 x_1 + \alpha_2 \vphi_2 y_1 - \alpha_0^2 \alpha_1^{-1} \vphi_1^{-1} x_0  = 0.
                                    \label{eq:s1}
\end{equation}
In \eqref{eq:1i} and \eqref{eq:3i} with $i \rightarrow i-1$, eliminate $z_i$ to get
\begin{align}
&
\begin{matrix}
  \alpha_i x_i  + \alpha_{i+1} \vphi_{i+1} y_i   - \alpha_{i-1}^2 \alpha_i^{-1} \vphi_i^{-1} x_{i-1}
\\
  \qquad\qquad\qquad\qquad\qquad\qquad
 - \alpha_{i-2} \alpha_{i-1} \alpha_i^{-1} \vphi_i^{-1} y_{i-2}
  = 0
\end{matrix}
     &&   (2 \leq i \leq d-1).               \label{eq:ss1}
\end{align}
In \eqref{eq:2i} for $i=1$ and \eqref{eq:30}, eliminate $z_1$ to get
\begin{equation}
  \alpha_2 x_1 + \alpha_3 \vphi_3 y_1    - \alpha_0 \vphi_2^{-1} x_0  = 0.   \label{eq:s2}
\end{equation}
In \eqref{eq:2i} and \eqref{eq:3i} with $i \rightarrow i-1$, eliminate $z_i$ to get
\begin{align}
      \alpha_{i+1} x_i  + \alpha_{i+2} \vphi_{i+2} y_i 
  - \alpha_{i-1} \vphi_{i+1}^{-1} x_{i-1} - \alpha_{i-2} \vphi_{i+1}^{-1} y_{i-2} &= 0      
                                           && (2 \leq i \leq d-2).             \label{eq:ss2}
\end{align}
We show that $x_i=0$ and $y_i=0$ for $0 \leq i \leq d-2$ using induction on $i$.
By Lemma \ref{lem:cond}(ii) $\alpha_1^2 \vphi_1 - \alpha_0 \alpha_2 \vphi_2 \neq 0$.
By this and  \eqref{eq:10}, \eqref{eq:20} we get $x_0 =0$ and $y_0=0$.
By this and \eqref{eq:s1}, \eqref{eq:s2},
\begin{align}
  \alpha_1 x_1 + \alpha_2 \vphi_2 y_1 &= 0,    \label{eq:aux1}
\\
 \alpha_2 x_1 + \alpha_3 \vphi_3 y_1 &= 0.     \label{eq:aux2}
\end{align}
By Lemma \ref{lem:cond}(ii) $\alpha_2^2 \vphi_2 - \alpha_1 \alpha_3 \vphi_3 \neq 0$.
By this and \eqref{eq:aux1}, \eqref{eq:aux2} we get $x_1=0$ and $y_1 = 0$.
Assume $2 \leq i \leq d-2$ and $x_j =0$, $y_j=0$ for $j < i$.
By  \eqref{eq:ss1} and \eqref{eq:ss2},
\begin{align}
  \alpha_i x_i + \alpha_{i+1} \vphi_{i+1} y_i &= 0,            \label{eq:aux1d}
\\
 \alpha_{i+1} x_i + \alpha_{i+2} \vphi_{i+2} y_i &= 0.         \label{eq:aux2d}
\end{align}
By Lemma \ref{lem:cond}(ii) $\alpha_{i+1}^2 \vphi_{i+1} - \alpha_i \alpha_{i+2} \vphi_{i+2} \neq 0$.
By this and \eqref{eq:aux1d}, \eqref{eq:aux2d} we get $x_i=0$ and $y_i=0$.
We have shown that $x_i=0$ and $y_i=0$ for $0 \leq i \leq d-2$.
By this and \eqref{eq:30}, \eqref{eq:3i} we get $z_i=0$ for $1 \leq i \leq d-1$.
By \eqref{eq:2i-1} we get $x_{d-1}=0$.
By \eqref{eq:1i} for $i=d-1$ we get $y_{d-1}=0$.
By \eqref{eq:2d} we get $z_d=0$.
By \eqref{eq:1d} we get $x_d=0$.
We have shown that $x_i = 0$ $(0 \leq i \leq d)$, $y_i = 0$ $(0 \leq i \leq d-1)$ and $z_i =0$ $(1 \leq i \leq d)$.
Thus $X=0$.
\end{proof}

\begin{lemma}   \label{lem:image}    \samepage
\ifDRAFT {\rm lem:image}. \fi
We have
\begin{align*}
 \dim {\cal X} E'_d & \leq 2,  &
 \dim E''_d {\cal X} & \leq 2,  &
 \dim {\cal X} A E'_d & \leq 3.
\end{align*}
\end{lemma}

\begin{proof}
Abbreviate ${\cal A} = \text{End}(V)$.
Observe that  
$\dim E'_i {\cal A} E'_j = 1$ for $0 \leq i,j \leq d$.
By the definition of $\cal X$, 
$E'_r {\cal X} E'_d = 0$ for $0 \leq r \leq d-2$.
Using this we argue
\[
 {\cal X} E'_d 
 = I {\cal X} E'_d
 = \sum_{r=0}^d E'_r X E'_d
= E'_{d-1} {\cal X} E'_d + E'_d {\cal X} E'_d.
\]
So
${\cal X} E'_d \subseteq E'_{d-1} {\cal A} E'_d + E'_d {\cal A} E'_d$,
and therefore $\dim {\cal X} E'_d \leq 2$.
Similarly $\dim E''_d {\cal X} \leq 2$.
By Lemma \ref{lem:trid},
$A E'_d = E'_{d-1} A E'_d + E'_d A E'_d$,
so
\[
   A E'_d \in E'_{d-1} {\cal A} E'_d + E'_d {\cal A} E'_d.
\]
Therefore
\[
  {\cal X} A E'_d \subseteq {\cal X} E'_{d-1} {\cal A} E'_d + {\cal X} E'_d {\cal A} E'_d.
\]
By this and the definition of $\cal X$,
\[
   {\cal X} A E'_d \subseteq 
    E'_{d-2} {\cal A} E'_d + E'_{d-1} {\cal A} E'_d + E'_d {\cal A} E'_d.
\]
Therefore $\text{dim} {\cal X} A E'_d \leq 3$.
\end{proof}

\begin{lemma}   \label{lem:main2}   \samepage
\ifDRAFT {\rm lem:main2}. \fi
The space $\cal X$ has dimension at most $7$.
\end{lemma}

\begin{proof}
Define linear maps
$\pi_1 : {\cal X} \to {\cal X} E'_d$ that sends $X \in {\cal X}$ to $X E'_d$,
$\pi_2 : {\cal X} \to E''_d {\cal X}$ that sends $X \in {\cal X}$ to  $E''_d X$,
$\pi_3 : {\cal X} \to {\cal X} A E'_{d}$ that sends $X \in {\cal X}$ to $X A E'_{d}$.
For $i=1,2,3$ let $K_i$ be the kernel of $\pi_i$.
By Lemma \ref{lem:image} the image of $\pi_1$ (resp.\ $\pi_2$) (resp.\ $\pi_3$)
has dimension at most $2$ (resp.\ $2$) (resp.\ $3$).
Therefore $K_1$ (resp.\ $K_2$) (resp.\ $K_3$) has codimension at most $2$
(resp.\ $2$) (resp.\ $3$).
By these comments the space $K_1 \cap K_2 \cap K_3$ has codimension at most $7$.
On the other hand, by Lemma \ref{lem:vanish} $K_1 \cap K_2 \cap K_3 = 0$.
Therefore $\cal X$ has  dimension at most $7$.
\end{proof}

\section{Proof of Theorem \ref{thm:main2}}
\label{sec:proofnonbipartite}

Let $A,B,C$ be an LR triple on $V$ 
with parameter array \eqref{eq:parray},
Toeplitz data \eqref{eq:Toeplitzdata}, trace data \eqref{eq:tracedata},
and idempotent data \eqref{eq:idempotentdata}.

\begin{lemma}  {\rm (See \cite[Proposition 14.1]{T:LRT}.) }   \samepage
\label{lem:ai}  
\ifDRAFT {\rm lem:ai}. \fi
For $0 \leq i \leq d$,
\[
  a_{d-i} = \alpha'_0 \beta'_1 \vphi''_i + \alpha'_1 \beta'_0 \vphi''_{i+1}.
\]
\end{lemma}

Assume $A,B,C$ is nonbipartite.
In view of Corollary \ref{cor:invariant}, assume that $A,B,C$ is normalized.
By Lemma \ref{lem:nonbipartitenormalized}
\begin{align}
  \alpha_i &= \alpha'_i = \alpha''_i && (0 \leq i \leq d),     \label{eq:equitable1}
\\
  \vphi_i &= \vphi'_i = \vphi''_i  &&  (1 \leq i \leq d).        \label{eq:equitable2}
\end{align}

\begin{lemma}   \label{lem:ai2}    \samepage
\ifDRAFT {\rm lem:ai2}. \fi
For $0 \leq i \leq d$,
\[
  a_i = \vphi_{d-i+1}- \vphi_{d-i}.
\]
\end{lemma}

\begin{proof}
By Lemma \ref{lem:ai} and \eqref{eq:equitable1}, \eqref{eq:equitable2},
\[
  a_i = \alpha_0 \beta_1 \vphi_i + \alpha_1 \beta_0 \vphi_{i+1}.
\]
By Lemma \ref{lem:13.46} $\alpha_0 = 1$ and $\beta_0 = 1$.
By Definition \ref{def:normalized1} and Lemma \ref{lem:13.46},
$\alpha_1 = 1$ and $\beta_1 = -1$.
Now the result follows from these comments.
\end{proof}

\begin{proofof}{Theorem \ref{thm:main2}}
In view of Corollary \ref{cor:invariant} we may assume that
$A,B,C$ is normalized. So $A,B,C$ is one of 
$\text{\rm NBG}_d(\F;q)$, $\text{\rm NBG}_d(\F;1)$, $\text{\rm NBNG}_d(\F;t)$.

First assume $d \geq 3$.
By Lemma \ref{lem:main2}, it suffices to show that the elements \eqref{eq:basis}
are linearly independent.
For scalars $e$, $f_1$, $f_2$, $f_3$, $g_1$, $g_2$, $g_3 \in \F$, define
\begin{equation}          \label{eq:defY}
 Y = e I + f_1 A + f_2 B + f_3 C + g_1 ABC + g_2 ACB + g_3 CAB.
\end{equation}
Assume $Y=0$, and we show that  $e$, $f_1$, $f_2$, $f_3$, $g_1$, $g_2$, $g_3$ are all $0$.
Fix an $(A,B)$-basis $\{v_i\}_{i=0}^d$ for $V$.
We identify each element of $\text{End}(V)$ with the matrix representing it with respect to $\{v_i\}_{i=0}^d$.
By Lemma \ref{lem:entriesABC} and \eqref{eq:equitable2},
 each of $A$, $B$, $C$ is tridiagonal with the following entries:
\[
 \begin{array}{ccc|ccc|ccc}
   A_{i,i-1} & A_{i,i} & A_{i-1,i} & B_{i,i-1} & B_{i,i} & B_{i-1,i} & C_{i,i-1} & C_{i,i} & C_{i-1,i}
 \\ \hline
  0 & 0 & 1 & \vphi_i & 0 & 0 & \vphi_{d-i+1} & a_i & \vphi_{d-i+1}/\vphi_i   \rule{0mm}{4mm}
 \end{array}
\]
Using these entries together with Lemma \ref{lem:ai2}, we compute the $(i,j)$-entry of $Y$ 
for $(i,j)=(0,0)$, $(0,1)$, $(1,0)$, $(1,1)$, $(1,2)$, $(2,1)$, $(2,3)$;
this yields a system of linear equations with unkowns  $e$, $f_1$, $f_2$, $f_3$, $g_1$, $g_2$, $g_3$.
The coefficient matrix $M$ is as follows:
\[
 \begin{array}{c|c|c|c|c|c|c|c}
  (i,j) & e & f_1 & f_2 & f_3 & g_1 & g_2 & g_3
 \\ \hline
 (0,0) &  1 & 0 & 0 &  -\vphi_d &
  -\vphi_1 \vphi_d & \vphi_1(\vphi_{d}-\vphi_{d-1}) & -\vphi_1 \vphi_d       \rule{0mm}{4mm}
\\
 (0,1) & 0 & 1 & 0 & \vphi_1^{-1}\vphi_d &
  \vphi_d & \vphi_{d-1} & \vphi_1^{-1}\vphi_2 \vphi_d
\\  \hline
 (1,0) & 0 & 0 & \vphi_1 & \vphi_d &
    \vphi_2 \vphi_d & \vphi_1 \vphi_{d-1} & \vphi_1 \vphi_d         \rule{0mm}{4mm}
\\
 (1,1) & 1 & 0 & 0 & \vphi_d- \vphi_{d-1} &
   \vphi_2 (\vphi_d - \vphi_{d-1}) & \vphi_2 (\vphi_{d-1}-\vphi_{d-2}) & \vphi_2 (\vphi_d-\vphi_{d-1})
\\
 (1,2) & 0 & 1 & 0 & \vphi_2^{-1} \vphi_{d-1} &
   \vphi_{d-1} & \vphi_{d-2} & \vphi_2^{-1} \vphi_3 \vphi_{d-1}
\\ \hline
 (2,1) & 0 & 0 & \vphi_2 & \vphi_{d-1} &
  \vphi_3 \vphi_{d-1} & \vphi_2 \vphi_{d-2} & \vphi_2 \vphi_{d-1}       \rule{0mm}{4mm}
\\
(2,3) & 0 & 1 & 0 & \vphi_3^{-1} \vphi_{d-2} &
   \vphi_{d-2} & \vphi_{d-3} & \vphi_3^{-1} \vphi_4 \vphi_{d-2}
\end{array}
\]
One routinely compute the determinant of $M$ for each of the cases
$\text{\rm NBG}_d(\F;q)$, $\text{\rm NBG}_d(\F;1)$, $\text{\rm NBNG}_d(\F;t)$:
\[
\begin{array}{c|c}
 \text{case} & \text{determinant of $M$} 
\\ \hline
\text{NBG}_d(\F;q) &      \rule{0mm}{8mm}
 \displaystyle
  \frac{(q^2-1)^2 (q^{d-1}-1)(q^d-1)(q^{d+1}+1)^3 }
         {q^{2d+3} (q-1)^4}
\\
\text{NBG}_d(\F;1)            \rule{0mm}{6mm}
&
 32 d(d-1)
\\
 \text{NBNG}_d(\F;t)        \rule{0mm}{6mm}
&
- t^{3(d+2)/2} (t-1)^2 (t^{d/2}-1)(t^{d+1}-1)^3
\end{array}
\]
In each case, the determinant of $M$ is nonzero, so 
 $e$, $f_1$, $f_2$, $f_3$, $g_1$, $g_2$, $g_3$ are all $0$.
The result follows.

Next assume $d=2$.
We proceed in a similar way as above.
By Lemma \ref{lem:dimXd2} it suffices to show that the elements
\eqref{eq:basisd2} are linearly independent.
Define $Y$ as in \eqref{eq:defY} with $g_3=0$.
We compute the $(i,j)$-entry of $Y$ for
 $(i,j)=(0,0)$, $(0,1)$, $(1,0)$, $(1,1)$, $(1,2)$, $(2,1)$;
this yields a system of linear equations with unknowns 
$e$, $f_1$, $f_2$, $f_3$, $g_1$, $g_2$.
The coefficient matrix $M'$ is obtained from $M$ by removing
the last row and the last column,
where we interpret $\vphi_{d-2}=0$ and $\vphi_3=0$.
One routinely compute the determinant of $M'$ for each of the cases
$\text{\rm NBG}_2(\F;q)$, $\text{\rm NBG}_2(\F;1)$:
\[
\begin{array}{c|c}
 \text{case} & \text{determinant of $M'$} 
\\ \hline
\text{NBG}_2(\F;q) &      \rule{0mm}{8mm}
 \displaystyle
  \frac{(q^2-1)^3 (q^3+1)^2 }
         {q^5 (q-1)^3}
\\
\text{NBG}_2(\F;1)            \rule{0mm}{6mm}
&
 32 
\end{array}
\]
In each case, the determinant of $M'$ is nonzero, so 
 $e$, $f_1$, $f_2$, $f_3$, $g_1$, $g_2$ are all $0$.
The result follows.
\end{proofof}

\section{Bipartite LR triples}
\label{sec:bipartite}

In this section, we recall the classification of bipartite LR triples.
Let $A,B,C$ be a bipartite LR triple on $V$ with parameter array \eqref{eq:parray},
Toeplitz data \eqref{eq:Toeplitzdata},
and idempotent data \eqref{eq:idempotentdata}.
To avoid triviality, we assume $d \geq 2$.

\begin{lemma}  {\rm (See \cite[Lemma 16.6]{T:LRT}.) }  \label{lem:alpha2}   \samepage
\ifDRAFT {\rm lem:alpha2}. \fi
The diameter $d$ is even.
Moreover for $0 \leq i \leq d$,
each of $\alpha_i$, $\alpha'_i$, $\alpha''_i$,
$\beta_i$, $\beta'_i$, $\beta''_i$
is zero if $i$ is odd and nonzero if $i$ is even.
\end{lemma}

By Lemma \ref{lem:alpha2} $d$ is even; set $m=d/2$.

\begin{lemma}  {\rm (See \cite[Lemma 16.12]{T:LRT}.) }  \label{lem:J}  \samepage
\ifDRAFT {\rm lem:J} \fi
\begin{itemize}
\item[\rm (i)]
The following spaces are equal:
\begin{align}
  & \sum_{j=0}^m E_{2j} V,   &
  & \sum_{j=0}^m E'_{2j} V,   &
  & \sum_{j=0}^m E''_{2j} V.                 \label{eq:Vout}
\end{align}
\item[\rm (ii)]
The following spaces are equal:
\begin{align}
  & \sum_{j=0}^{m-1} E_{2j+1} V,   &
  & \sum_{j=0}^{m-1} E'_{2j+1} V,   &
  & \sum_{j=0}^{m-1} E''_{2j+1} V.                 \label{eq:Vin}
\end{align}
\end{itemize}
\end{lemma}

Let $V_\text{out}$ and $V_\text{in}$ denote the common spaces of 
\eqref{eq:Vout} and \eqref{eq:Vin}, respectively.
Then
\[
  V = V_\text{out} + V_\text{in}  \qquad \text{(direct sum)}.
\]

\begin{lemma}   {\rm (See \cite[Lemma 16.15]{T:LRT}.)}
\label{lem:16.15}   \samepage
\ifDRAFT {\rm lem:16.15}. \fi
We have
\begin{align*}
  A V_\text {\rm out} &= V_\text{\rm in},  &
  B V_\text{\rm out} &= V_\text{\rm in},  &
  C V_\text{\rm out} &= V_\text{\rm in},  &
\\
  A V_\text{\rm in} &\subseteq V_\text{\rm out},  &
  B V_\text{\rm in} &\subseteq V_\text{\rm out},  &
  C V_\text{\rm in} &\subseteq V_\text{\rm out}.
\end{align*}
\end{lemma}

\begin{definition}   {\rm (See \cite[Definition 16.29]{T:LRT}.)}
\label{def:Aout}  \samepage
\ifDRAFT {\rm def:Aout}. \fi
Define $A_\text{out}$, $A_\text{in}$, $B_\text{out}$, $B_\text{in}$,
$C_\text{out}$, $C_\text{in}$ in $\text{End}(V)$ as follows.
The map $A_\text{out}$ acts on $V_\text{out}$ as $A$, 
and on $V_\text{in}$ as zero.
The map $A_\text{in}$ acts on $V_\text{in}$ as $A$,
and on $V_\text{out}$ as zero.
The other maps are similarly defined.
\end{definition}

Observe
\begin{align*}
 A &= A_\text{out} + A_\text{in}, &
 B &= B_\text{out} + B_\text{in}, &
 C &= C_\text{out} + C_\text{in}.
\end{align*}

\begin{lemma}  {\rm (See \cite[Lemma 16.31, 16.33]{T:LRT}.) }   \label{lem:Bscalar}  \samepage
\ifDRAFT {\rm lem:Bscalar}. \fi
For nonzero scalars 
$\alpha_\text{\rm out}$, $\alpha_\text{\rm in}$,
$\beta_\text{\rm out}$, $\beta_\text{\rm in}$,
$\gamma_\text{\rm out}$, $\gamma_\text{\rm in}$ in $\F$,
the sequence
\begin{align}
 & \alpha_\text{\rm out} A_\text{\rm out} + \alpha_\text{\rm in} A_\text{\rm in},   
&&  \beta_\text{\rm out} B_\text{\rm out} + \beta_\text{\rm in} B_\text{\rm in},   
&&  \gamma_\text{\rm out} C_\text{\rm out} + \gamma_\text{\rm in} C_\text{\rm in}
\end{align}
is a bipartite LR triple on $V$.
Moreover this LR triple has the same idempotent data as $A,B,C$.
\end{lemma}

\begin{definition}   {\rm (See \cite[Definition 16.36]{T:LRT}.)}  
\label{def:biassociate} \samepage
\ifDRAFT {\rm def:biassociate}. \fi
Two bipartite LR triples $A,B,C$ and $A',B',C'$ are called {\em biassociate}
whenever there exist nonzero scalars $\alpha$, $\beta$, $\gamma$
such that
\begin{align*}
 A' &= \alpha A_\text{out} + A_\text{in}, &
 B' &= \beta  B_\text{out} + B_\text{in}, &
 C' &= \gamma  C_\text{out} + C_\text{in}.
\end{align*}
\end{definition}

\begin{lemma}   \label{lem:Xbisimilar}   \samepage
\ifDRAFT {\rm lem:Xbisimilar}. \fi
Biassociate bipartite LR triples have the same tridiagonal space.
\end{lemma}

\begin{proof}
By Lemma \ref{lem:Bscalar} biassociate LR triples have the same idempotent data.
The result follows.
\end{proof}

In the direct sum $V = V_\text{out} + V_\text{in}$,
let $J \in \text{End}(V)$ denote the projection onto $V_\text{out}$.
Then $I-J$ is the projection onto $V_\text{in}$.
Observe $J^2 = J$ and $J(I-J)=0 = (I-J)J$.
Also $V_\text{out} = J V$ and $V_\text{in} = (I-J)V$.

\begin{lemma}  {\rm (See \cite[Lemma 16.30]{T:LRT}.)}
\label{lem:16.30}  \samepage
\ifDRAFT {\rm lem:16.30}. \fi
We have
\begin{align*}
 A_\text{\rm out} &= A J = (I-J) A,  &
 A_\text{\rm in}   &= J A = A (I-J),
\\
 B_\text{\rm out} &= B J = (I-J) B,  &
 B_\text{\rm in}   &= J B = B (I-J),
\\
 C_\text{\rm out} &= C J = (I-J) C,  &
 C_\text{\rm in}   &= J C = C (I-J).
\end{align*}
\end{lemma}

\begin{lemma}    \label{lem:bipartiteABC}   \samepage
\ifDRAFT {\rm lem:bipartiteABC}. \fi
We have
\begin{align*}
  A B C J &= A_\text{\rm out} B_\text{\rm in} C_\text{\rm out},  &
  A C B J &= A_\text{\rm out} C_\text{\rm in} B_\text{\rm out},
\\
  B C A J &= B_\text{\rm out} C_\text{\rm in} A_\text{\rm out},  &
  B A C J &= B_\text{\rm out} A_\text{\rm in} C_\text{\rm out},
\\
  C A B J &= C_\text{\rm out} A_\text{\rm in} B_\text{\rm out},  &
  B A C J &= C_\text{\rm out} B_\text{\rm in} A_\text{\rm out}.
\end{align*}
\end{lemma}

\begin{proof}
Use $J^2=J$, $(I-J)^2 = I-J$ and Lemma \ref{lem:16.30}.
\end{proof}

\begin{lemma}   \label{lem:bipartitescalar}   \samepage
\ifDRAFT {\rm lem:bipartitescalar}. \fi
For nonzero scalars 
$\alpha_\text{\rm out}$, $\alpha_\text{\rm in}$,
$\beta_\text{\rm out}$, $\beta_\text{\rm in}$,
$\gamma_\text{\rm out}$, $\gamma_\text{\rm in}$ in $\F$,
define
\begin{align*}
  A' &= \alpha_\text{\rm out} A_\text{\rm out} + \alpha_\text{\rm in} A_\text{\rm in},  &
  B' &= \beta_\text{\rm out} B_\text{\rm out} + \beta_\text{\rm in} B_\text{\rm in},  &
  C' &= \gamma_\text{\rm out} C_\text{\rm out} + \gamma_\text{\rm in} C_\text{\rm in}.
\end{align*}
Then $A' B' C' J$ is a nonzero scalar multiple of $A B C J$,
and $A' C' B' J$ is a nonzero scalar multiple of $A C B J$.
\end{lemma}

\begin{proof}
By Lemma \ref{lem:bipartiteABC}
$A' B' C' J = A'_\text{out} B'_\text{in} C'_\text{out}$.
Observe $A_\text{out} J = A_\text{out}$ and $A_\text{in} J = 0$.
Using this we argue
\[
 A'_\text{out} = A' J = \alpha_\text{out} A_\text{out} J + \alpha_\text{in} A_\text{in} J 
   = \alpha_\text{out} A_\text{out}.
\]
So $A'_\text{out} = \alpha_\text{out} A_\text{out}$.
Similarly
$B'_\text{in} = \beta_\text{in} B_\text{in}$,
$C'_\text{out} = \gamma_\text{out} C_\text{out}$.
By these comments
\[
  A' B' C' J = 
   \alpha_\text{out} \beta_\text{in} \gamma_\text{out} 
   A_\text{out} B_\text{in} C_\text{out}.
\]
By this and Lemma \ref{lem:bipartiteABC} 
$A' B' C' J =   
   \alpha_\text{out} \beta_\text{in} \gamma_\text{out} 
   A B C J$.
So $A' B' C' J$ is a nonzero scalar multiple of $A B C J$.
The proof is similar for $A' C' B' J$.
\end{proof}

We recall the normalization of a bipartite LR triple.

\begin{definition}    {\rm See \cite[Definition 18.11]{T:LRT}.)}
\label{def:Bnormalized}   \samepage
\ifDRAFT {\rm def:Bnormalized}. \fi
The LR triple $A,B,C$ is said to be {\em normalized} whenever
$\alpha_2 = 1$, $\alpha'_2=1$, $\alpha''_2=1$.
\end{definition}

\begin{lemma}   {\rm See \cite[Corollary 18.15]{T:LRT}.)}
\label{lem:Bnormalized}   \samepage
\ifDRAFT {\rm lem:Bnormalized}. \fi
There exists a unique sequence of nonzero scalars $\alpha$, $\beta$, $\gamma$ in $\F$
such that the LR triple
$\alpha A_\text{out} + A_\text{in}$, 
$\beta  B_\text{out} + B_\text{in}$,  
$\gamma  C_\text{out} + C_\text{in}$
is normalized.
\end{lemma}

\begin{lemma}   {\rm (See \cite[Lemma 18.12]{T:LRT}.)}   \label{lem:equitableB} \samepage
\ifDRAFT {\rm lem:equitableB}. \fi
Assume $A,B,C$ is bipartite and normalized.
Then $\alpha_i = \alpha'_i = \alpha''_i$ and
$\beta_i = \beta'_i = \beta''_i$ for $0 \leq i \leq d$.
\end{lemma}

Assume $A,B,C$ is normalized.
By the classification in \cite[Section 39]{T:LRT}, $A,B,C$ is isomorphic to
one of the following LR triples.

\begin{definition}    {\rm (See \cite[Example 30.1]{T:LRT}.) }
\label{def:Bdt}   \samepage
\ifDRAFT {\rm def:Bdt}. \fi
The LR triple $\text{B}_d(\F; t, \rho_0, \rho'_0, \rho''_0)$ is over $\F$,
diameter $d$, bipartite, normalized,
and satisfies
\begin{align*}
  & d \geq 4;  \qquad 
  \text{$d$ is even}; \qquad
  0 \neq t \in \F;  \qquad
  t^i \neq 1 \quad (1 \leq i \leq d/2);
\\
 & \rho_0, \rho'_0, \, \rho''_0 \in \F;  \qquad\qquad
    \rho_0 \rho'_0 \rho''_0 = - t^{1-d/2};
\\
 & \vphi_i =
     \begin{cases}
       \rho_0 \frac{ 1- t^{i/2} }{ 1-t}  &  \text{ if $i$ is even},   \\
      \frac{ t }{ \rho_0 } \frac{ 1 - t^{(i-d-1)/2} }{ 1-t } & \text{ if $i$ is odd}
     \end{cases}
    \qquad\qquad (1 \leq i \leq d);
\\
 & \vphi'_i =
     \begin{cases}
       \rho'_0 \frac{ 1- t^{i/2} }{ 1-t}  &  \text{ if $i$ is even},   \\
      \frac{ t }{ \rho'_0 } \frac{ 1 - t^{(i-d-1)/2} }{ 1-t } & \text{ if $i$ is odd}
     \end{cases}
    \qquad\qquad (1 \leq i \leq d);
\\
 & \vphi''_i =
     \begin{cases}
       \rho''_0 \frac{ 1- t^{i/2} }{ 1-t}  &  \text{ if $i$ is even},   \\
      \frac{ t }{ \rho''_0 } \frac{ 1 - t^{(i-d-1)/2} }{ 1-t } & \text{ if $i$ is odd}
     \end{cases}
    \qquad\qquad (1 \leq i \leq d).
\end{align*}
\end{definition}

\begin{definition}    {\rm (See \cite[Example 30.2]{T:LRT}.) }
\label{def:Bd1}   \samepage
\ifDRAFT {\rm def:Bd1}. \fi
The LR triple $\text{B}_d(\F; 1, \rho_0, \rho'_0, \rho''_0)$ is over $\F$,
diameter $d$, bipartite, normalized,
and satisfies
\begin{align*}
  & d \geq 4;  \qquad 
  \text{$d$ is even}; \qquad
  \text{$\text{Char}(\F)$ is $0$ or greater that $d/2$};
\\
 & \rho_0, \rho'_0, \, \rho''_0 \in \F;  \qquad\qquad
    \rho_0 \rho'_0 \rho''_0 = - 1;
\\
 & \vphi_i =
     \begin{cases}
       \frac{ i \rho_0}{ 2}  &  \text{ if $i$ is even},   \\
      \frac{ i-d-1 }{ 2 \rho_0 }  & \text{ if $i$ is odd}
     \end{cases}
    \qquad\qquad (1 \leq i \leq d);
\\
 & \vphi'_i =
     \begin{cases}
       \frac{ i \rho'_0}{ 2}  &  \text{ if $i$ is even},   \\
      \frac{ i-d-1 }{ 2 \rho'_0 }  & \text{ if $i$ is odd}
     \end{cases}
    \qquad\qquad (1 \leq i \leq d);
\\
 & \vphi''_i =
     \begin{cases}
       \frac{ i \rho''_0}{ 2}  &  \text{ if $i$ is even},   \\
      \frac{ i-d-1 }{ 2 \rho''_0 }  & \text{ if $i$ is odd}
     \end{cases}
    \qquad\qquad (1 \leq i \leq d).
\end{align*}
\end{definition}

\begin{definition}    {\rm (See \cite[Example 30.2]{T:LRT}.) }
\label{def:B2}   \samepage
\ifDRAFT {\rm def:B2}. \fi
The LR triple $\text{B}_2(\F; \rho_0, \rho'_0, \rho''_0)$ is over $\F$,
diameter $2$, bipartite, normalized,
and satisfies
\begin{align*}
 & \rho_0, \rho'_0, \, \rho''_0 \in \F;  \qquad\qquad
    \rho_0 \rho'_0 \rho''_0 = -1;
\\
 & \vphi_1 = -1 / \rho_0,  \qquad
    \vphi'_1 = -1 / \rho'_0,  \qquad
    \vphi''_1 = -1 / \rho''_0,
\\
 & \vphi_2 = \rho_0,  \qquad
   \vphi'_2 = \rho'_0, \qquad
   \vphi''_2 = \rho''_0.
\end{align*}
\end{definition}

\section{Bounding the dimension of the tridiagonal space; bipartite case}
\label{sec:boundbipartite}

Let $A,B,C$ be a bipartite LR triple on $V$ with parameter array \eqref{eq:parray},
Toeplitz data \eqref{eq:Toeplitzdata},
and idempotent data \eqref{eq:idempotentdata}.
Let ${\cal X}$ be the tridiagonal space for $A,B,C$.

\begin{lemma} {\rm (See \cite[Proposition 31.3]{T:LRT})}  \label{lem:alphabeta}   \samepage
\ifDRAFT {\rm lem:alphabeta}. \fi
The following hold:
\begin{itemize}
\item[\rm (i)]
For the LR triple $\text{\rm B}_d(\F; t, \rho_0, \rho'_0, \rho''_0)$,
\begin{align*}
\alpha_{2i} &= \frac{(1-t)^i}
                          {(t;t)_i},
&
\beta_{2i} &= 
 \frac{(-1)^i t^{i(i-1)/2} (1-t)^i}
        { (t;t)_i }
&& (0 \leq i \leq d/2).
\end{align*}
\item[\rm (ii)]
For the LR triple $\text{\rm B}_d(\F; 1, \rho_0, \rho'_0, \rho''_0)$,
\begin{align*}
\alpha_{2i} &= \frac{1}
                          {i !},
&
\beta_{2i} &= 
 \frac{(-1)^i }
        { i ! }
&& (0 \leq i \leq d/2).
\end{align*}
\end{itemize}
\end{lemma}

Assume $A,B,C$ is normalized.
We may assume that $A,B,C$ is one of
$\text{B}_d(\F; t, \rho_0, \rho'_0, \rho''_0)$,
 $\text{B}_d(\F; 1, \rho_0, \rho'_0, \rho''_0)$,
$\text{B}_2(\F;  \rho_0, \rho'_0, \rho''_0)$.

\begin{lemma}  \label{lem:condbipartite}   \samepage
\ifDRAFT {\rm lem:condbipartite}. \fi
We have
\begin{align*}
 \alpha_{2i}^2 &\neq \alpha_{2i-2} \alpha_{2i+2},
&
 \beta_{2i}^2 &\neq \beta_{2i-2} \beta_{2i+2}
&& (1 \leq i \leq d/2-1).
\end{align*}
\end{lemma}

\begin{proof}
Use Lemma \ref{lem:alphabeta}.
\end{proof}

\begin{lemma}   \label{lem:mainB}   \samepage
\ifDRAFT {\rm lemma:mainB}. \fi
Assume $d \geq 4$.
Then the vector space $\cal X$ has dimension at most $8$.
\end{lemma}

\begin{proof}
Fix an $(A,B)$-basis $\{v_i\}_{i=0}^d$ for $V$.
We identify each element of $\text{End}(V)$ with the matrix that represents
it with respect to $\{v_i\}_{i=0}^d$.
Pick an element $X \in {\cal X}$.
Then $X$ is tridiagonal; set
\[
 X = 
 \begin{pmatrix}
  x_0 & y_0 &  & & & \text{\bf 0}  \\
  z_1 & x_1 & y_1 \\
       & z_2 & \cdot & \cdot  \\
       &      & \cdot & \cdot & \cdot  \\
       &      &          & \cdot & \cdot & y_{d-1} \\
  \text{\bf 0}   &      &          &          & z_d & x_d
 \end{pmatrix}.
\]
For notational convenience, define $y_{-1}=0$, $y_d=0$, $z_0=0$, $z_{d+1}=0$.
By Proposition \ref{prop:Eddrij} and Lemma \ref{lem:equitableB},
\begin{align*}
 (E''_r)_{i,j} &=
  \begin{cases}
    \alpha_{r-d+j} \beta_{d-r-i}  & \text{ if $i \leq d-r \leq j$},
 \\
    0  &  \text{ otherwise}
  \end{cases}
  &&  (0 \leq i,j \leq d).
\end{align*}
For $0 \leq r \leq d-2$, compute the $(0,d)$-entry of $E''_{d-r} X E''_{d-r-2}$ if $r$ is even,
and compute the $(1,d-1)$-entry of $E''_{d-r} X E''_{d-r-2}$ if $r$ is odd.
This yields
\begin{align}
 \alpha_0 \beta_2 x_r + \alpha_2 \beta_0 x_{r+2} = 0
    && (0 \leq r \leq d-2).              \label{eq:Bxr}
\end{align}
For $0 \leq r \leq d-3$ we compute the $(0,d-1)$-entry of $E''_{d-r} X E''_{d-r-3}$ if $r$ is even,
and $(1,d)$-entry of $E''_{d-r} X E''_{d-r-3}$ if $r$ is odd.
This yields
\begin{align}
 & \alpha_0 \beta_2 y_r + \alpha_2 \beta_0 y_{r+2}
       + \alpha_0 \beta_4 z_r + \alpha_2 \beta_2 z_{r+2} + \alpha_4 \beta_0 z_{r+4} = 0
        && (0 \leq r \leq d-3).                               \label{eq:Bequ1}
\end{align}
When $d \geq 6$,
for $0 \leq r \leq d-5$ we compute the $(0,d-1)$-entry of $E''_{d-r} X E''_{d-r-5}$ 
if $r$ is even,
and $(1,d)$-entry of $E''_{d-r} X E''_{d-r-5}$ if $r$ is odd.
This yields
\begin{align}
 & \alpha_0 \beta_4 y_r + \alpha_2 \beta_2 y_{r+2} + \alpha_4 \beta_0 y_{r+4}  \notag
\\
 & \qquad\qquad
   + \alpha_0 \beta_6 z_r + \alpha_2 \beta_4 z_{r+2}
  + \alpha_4 \beta_2 z_{r+4} + \alpha_6 \beta_0 z_{r+6} = 0
  &&  (0 \leq r \leq d-5).                                         \label{eq:Bequ2}
\end{align}
We show that each entry of $X$ are uniquely determined by
$x_0$, $x_1$, $y_0$, $y_{d-1}$, $z_2$, $z_4$, $z_{d-3}$, $z_{d-1}$.
To this aim, we assume
\begin{equation}
  x_0 = 0, \quad
  x_1 = 0, \quad
  y_0 = 0, \quad
  y_{d-1} = 0, \quad
  z_2 = 0, \quad
  z_4 = 0, \quad
  z_{d-3} = 0, \quad
  z_{d-1} = 0,                                   \label{eq:assump}
\end{equation}
and we show $X=0$.
By \eqref{eq:Bxr} and $x_0=0$, $x_1=0$ we find $x_r=0$ for $0 \leq r \leq d$.
We show that
\begin{align}
  y_{d-r} &= 0, \qquad\qquad z_{d-r}=0    &&  r=1,3,5, \ldots, d-1.    \label{eq:claim1}
\end{align}
By \eqref{eq:assump} $y_{d-1}=0$ and $z_{d-1}=0$, so \eqref{eq:claim1}
holds for $r=1$.
By \eqref{eq:Bequ1} for $r=d-3$,
\[
   \alpha_0 \beta_2 y_{d-3} + \alpha_2 \beta_0 y_{d-1}
   + \alpha_0 \beta_4 z_{d-3} + \alpha_2 \beta_2 z_{d-1} = 0.
\]
By \eqref{eq:assump} $z_{d-3}=0$, $z_{d-1}=0$, $y_{d-1}=0$.
By these comments $y_{d-3}=0$.
So \eqref{eq:claim1} holds for $r=3$.
Now assume $r \geq 5$, $r$ is odd.
By induction and \eqref{eq:Bequ1},
\begin{equation}
   \alpha_0 \beta_2 y_r + \alpha_0 \beta_4 z_r = 0.             \label{eq:Baux1}
\end{equation}
By induction and \eqref{eq:Bequ2},
\begin{equation}
   \alpha_0 \beta_4 y_r + \alpha_0 \beta_6 z_r = 0.             \label{eq:Baux2}
\end{equation}
By Lemma \ref{lem:condbipartite} $\beta_2 \beta_6 - \beta_4^2 \neq 0$,
so we can solve the system of linear equations \eqref{eq:Baux1}, \eqref{eq:Baux2};
This yields $y_r = 0$ and $z_r =0$.
We have shown \eqref{eq:claim1}.
Next we show
\begin{align}
 y_r &= 0, \qquad\qquad z_{r+2}=0
  &&  r=0,2,4,\ldots,d-2.                              \label{eq:claim2}
\end{align}
By \eqref{eq:assump} $y_0=0$ and $z_2=0$. So \eqref{eq:claim2} holds for $r=0$.
By \eqref{eq:Bequ1} for $r=0$,
\[
  \alpha_0 \beta_2 y_0 + \alpha_2 \beta_0 y_2
  + \alpha_2 \beta_2 z_2 + \alpha_4 \beta_0 z_4 = 0.
\]
By this and \eqref{eq:assump} $y_2 = 0$.
So \eqref{eq:claim2} holds for $r=2$.
Now assume $r \geq 4$, $r$ is even.
By  \eqref{eq:Bequ1}
\[
  \alpha_0 \beta_2 y_r + \alpha_2 \beta_0 y_{r+2}
  + \alpha_0 \beta_4 z_r + \alpha_2 \beta_2 z_{r+2} + \alpha_4 \beta_0 z_{r+4} = 0.
\]
By induction, $y_r=0$, $z_r=0$, $z_{r+2}=0$.
By these comments
\begin{equation}
  \alpha_2 \beta_0 y_{r+2} + \alpha_4 \beta_0 z_{r+4}=0.    \label{eq:Baux3}
\end{equation}
By \eqref{eq:Bequ2} with $r \rightarrow r-2$,
\[
  \alpha_0 \beta_4 y_{r-2} + \alpha_2 \beta_2 y_r + \alpha_4 \beta_0 y_{r+2}
 + \alpha_0 \beta_6 z_{r-2} + \alpha_2 \beta_4 z_r 
 + \alpha_4 \beta_2 z_{r+2} + \alpha_6 \beta_0 z_{r+4} = 0.
\]
By induction $y_{r-2}=0$, $y_r=0$, $z_{r-2}=0$, $z_r=0$, $z_{r+2}=0$.
By these comments
\begin{equation}
  \alpha_4 \beta_0 y_{r+2} + \alpha_6 \beta_0 z_{r+4} = 0.   \label{eq:Baux4}
\end{equation}
By Lemma \ref{lem:condbipartite} $\alpha_2 \alpha_6 - \alpha_4^2 \neq 0$,
so we can solve the linear equations \eqref{eq:Baux3}, \eqref{eq:Baux4}.
This yields $y_{r+2}=0$ and $z_{r+4}=0$.
We have shown \eqref{eq:claim2}.
By \eqref{eq:assump}, \eqref{eq:claim1}, \eqref{eq:claim2} 
we get $y_r=0$ for $0 \leq r \leq d-1$ and $z_r = 0$ for $1 \leq r \leq d$.
So $X=0$.
Thus $\cal X$ has dimension at most $8$.
\end{proof}

\section{Proof of Theorem \ref{thm:main1}}
\label{sec:proofbipartite}

Let $A,B,C$ be an LR triple on $V$
with parameter array \eqref{eq:parray}
and Toeplitz data \eqref{eq:Toeplitzdata}.

\begin{lemma}  {\rm (See \cite[Proposition 14.6]{T:LRT}.) }    \label{lem:prop14.6}  \samepage
\ifDRAFT {\rm lem:prop14.6}. \fi
For $2 \leq i \leq d-1$,
\begin{align*}
\frac{ \vphi'_i } {\vphi''_{d-i+1} }
 &= \alpha'_0 \beta'_2 \vphi_{i-1} + \alpha'_1 \beta'_1 \vphi_i 
                 + \alpha'_2 \beta'_0 \vphi_{i+1},
\\
\frac{ \vphi''_i } {\vphi_{d-i+1} }
 &= \alpha''_0 \beta''_2 \vphi'_{i-1} + \alpha''_1 \beta''_1 \vphi'_i 
                 + \alpha''_2 \beta''_0 \vphi'_{i+1},
\\
\frac{ \vphi_i } {\vphi'_{d-i+1} }
 &= \alpha_0 \beta_2 \vphi''_{i-1} + \alpha_1 \beta_1 \vphi''_i 
                 + \alpha_2 \beta_0 \vphi''_{i+1}.
\end{align*}
\end{lemma}

\begin{lemma}   {\rm (  \cite[Lemma 16.9]{T:LRT}.) }  \label{lem:16.9}  \samepage
\ifDRAFT {\rm lem:16.9}. \fi
Assume $A,B,C$ is bipartite and $d \geq 2$. 
Then $\beta_2 = - \alpha_2$, $\beta'_2 = - \alpha'_2$, $\beta''_2 = - \alpha''_2.$
\end{lemma}

\begin{lemma}   \label{lem:vphii-1i+1}   \samepage
\ifDRAFT {\rm lem:vphii-1i+1}. \fi
Assume $A,B,C$ is bipartite.
Then 
\begin{align*}
  \vphi_{i-1} &\neq \vphi_{i+1}, &
 \vphi'_{i-1} &\neq \vphi'_{i+1}, &
 \vphi''_{i-1} &\neq \vphi''_{i+1}
&& (2 \leq i \leq d-1).
\end{align*}
\end{lemma}

\begin{proof}
We first show $\vphi_{i-1} \neq \vphi_{i+1}$.
By Lemma \ref{lem:13.46}  $\alpha'_0 = 1$ and $\beta'_0 =1$.
By  Lemma \ref{lem:alpha2} $\alpha'_1=0$ and  $\beta'_1=0$.
By Lemma \ref{lem:16.9} $\beta'_2 = - \alpha'_2$.
By these comments and Lemma \ref{lem:prop14.6}
$\beta'_2 (\vphi_{i-1} - \vphi_{i+1}) = \vphi'_i / \vphi''_{d-i+1}$.
By Lemma \ref{lem:alpha2} $\beta'_2 \neq 0$.
By these comments $\vphi_{i-1} \neq \vphi_{i+1}$.
In a similar way, we can show that $\vphi'_{i-1} \neq \vphi'_{i+1}$
and $\vphi''_{i-1} \neq \vphi''_{i+1}$.
\end{proof}

Let  ${\cal X}$ be the tridiagonal space for $A,B,C$.

\begin{lemma}    \label{lem:contained}   \samepage
\ifDRAFT {\rm lem:contained}. \fi
Assume $A,B,C$ is bipartite.
Then for $X \in {\cal X}$, the elements $XJ$ and $X(I-J)$ are contained in $\cal X$.
\end{lemma}

\begin{proof}
Pick any $X \in {\cal X}$.
Pick any integers $r,s$ such that $0 \leq r,s \leq d$ and $|r-s|>1$.
We show that $E_r X J E_s=0$.
By \eqref{eq:defJ} $J = \sum_{i=0}^{d/2} E_{2i}$.
So $J E_s = 0$ if $s$ is odd, and $J E_s = E_s$ if $s$ is even.
We have $E_r X E_s = 0$ since $X \in {\cal X}$.
By these comments $E_r X J E_s = 0$.
In a similar way, $E'_r X J E'_s = 0$ and $E''_r X J E''_s = 0$.
Thus $X J$ is contained in $\cal X$.
Similarly, $X (I-J)$ is contained in $\cal X$.
\end{proof}

\begin{proofof}{Theorem \ref{thm:main1}}
In view of Lemmas \ref{lem:Xbisimilar} and \ref{lem:bipartitescalar}, 
we may assume that $A,B,C$ is normalized.

Note by Lemma \ref{lem:contained} that ${\cal X} J \subseteq {\cal X}$
and ${\cal X} (I-J) \subseteq {\cal X}$.
Observe $X = X J + X (I-J)$ for $X \in {\cal X}$.
Therefore ${\cal X} = {\cal X} J + {\cal X} (I-J)$.
Using $J^2 = J$ and $(I-J) J = 0$,
one checks that spaces ${\cal X}J$ and ${\cal X}(I-J)$ has zero intersection.
So ${\cal X} = {\cal X} J + {\cal X} (I-J)$ is a direct sum.

Fix an $(A,B)$-basis $\{v_i\}_{i=0}^d$ for $V$.
We identify each element of $\text{End}(V)$ with the matrix representing it
with respect to $\{v_i\}_{i=0}^d$.
By Lemma \ref{lem:entriesABC} and since $A,B,C$ is bipartite,
each of $A$, $B$, $C$ is tridiagonal with the following entries:
\[
 \begin{array}{ccc|ccc|ccc}
   A_{i,i-1} & A_{i,i} & A_{i-1,i} & B_{i,i-1} & B_{i,i} & B_{i-1,i} & C_{i,i-1} & C_{i,i} & C_{i-1,i}
 \\ \hline
  0 & 0 & 1 & \vphi_i & 0 & 0 & \vphi''_{d-i+1} & 0 & \vphi'_{d-i+1}/\vphi_i   \rule{0mm}{4mm}
 \end{array}
\]
The matrix $J$ is the diagonal matrix whose $(i,i)$-entry is $1$ if $i$ is even,
and $0$ if $i$ is odd.

(i):
Assume $d=2$.
We have
\begin{align*}
J &=
  \begin{pmatrix}
    1 & 0 & 0  \\
    0 & 0 & 0  \\
    0 & 0 & 1
  \end{pmatrix},
&
A &=
  \begin{pmatrix}
    0 & 1 & 0 \\
    0 & 0 & 1 \\
    0 & 0 & 0
  \end{pmatrix},
&
B &=
  \begin{pmatrix}
    0 & 0 & 0 \\
    \vphi_1 & 0 & 0 \\
    0 & \vphi_2 & 0
  \end{pmatrix}.
\end{align*}
Now one routinely checks that $J$, $AJ$, $BJ$ are linearly independent.
Similarly $I-J$, $A(I-J)$, $B(I-J)$ are linearly independent.
By Lemma \ref{lem:dimXd2} $\dim {\cal X} \leq 6$.
The result follows.

(ii):
Assume $d \geq 4$.
We first show that $J$, $AJ$, $BJ$, $ACBJ$ are linearly independent.
For scalars $f_0$, $f_1$, $f_2$, $f_3$ in $\F$, set
\[
  Y = f_0 J + f_1 A J + f_2 B J + f_3 ACB J.
\]
We assume $Y=0$, and we show   $f_0$, $f_1$, $f_2$, $f_3$ are all $0$.
Compute the $(0,0)$-entry of $Y$ to find $f_0 = 0$.
Compute the $(i,j)$-entry of $Y$ for $(i,j)=(1,0)$, $(1,2)$, $(3,2)$ to get
\begin{align*}
 0 &=  f_2 +  \vphi''_{d-1}  f_3,
\\
 0 &= f_1 + \vphi'_{d-2} f_3,
\\
 0 &= f_2 +  \vphi''_{d-3} f_3.
\end{align*}
Viewing the above equations as a system of linear equations with unknowns
$f_1$, $f_2$, $f_3$, let $M$ be the coefficient matrix.
Then
\[
M =
 \begin{pmatrix}
 0 & 1 &  \vphi''_{d-1}
\\
 1 & 0 & \vphi'_{d-2}
\\
 0 & 1 & \vphi''_{d-3}
\end{pmatrix}.
\]
The determinant of $M$ is
\[
  \det M = \vphi''_{d-1} - \vphi''_{d-3}.
\]
This is nonzero by Lemma \ref{lem:vphii-1i+1},
so $f_1$, $f_2$, $f_3$ are all $0$.
We have shown that $J$, $AJ$, $BJ$, $ACBJ$
are linearly independent.

Next we show that $I-J$, $A(I-J)$, $B(I-J)$, $ABC(I-J)$ are linearly independent.
For scalars $f'_0$, $f'_1$, $f'_2$, $f'_3$ in $\F$, set
\[
  Y' = f_0 (I-J) + f'_1 A (I-J) + f'_2 B (I-J) + f'_3 ABC (I-J).
\]
We assume $Y'=0$, and we show   $f'_0$, $f'_1$, $f'_2$, $f'_3$ are all $0$.
Compute the $(1,1)$-entry of $Y'$ to find $f'_0 = 0$.
Compute the $(i,j)$-entry of $Y'$ for $(i,j)=(0,1)$, $(2,1)$, $(2,3)$ to get
\begin{align*}
 0 &=  f'_1 +  \vphi'_{d}  f'_3,
\\
 0 &= \vphi_2 f'_2 + \vphi_3 \vphi''_{d-1} f'_3,
\\
 0 &= f'_1 +  \vphi'_{d-2} f'_3.
\end{align*}
Viewing the above equations as a system of linear equations with unknowns
$f'_1$, $f'_2$, $f'_3$, let $M'$ be the coefficient matrix.
Then
\[
M' =
 \begin{pmatrix}
 1 & 0 &  \vphi'_{d}
\\
 0 & \vphi_2 & \vphi_3 \vphi''_{d-1}
\\
 1 & 0 & \vphi'_{d-2}
\end{pmatrix}.
\]
The determinant of $M'$ is
\[
  \det M' = \vphi_2 (\vphi'_{d-2} - \vphi'_{d}).
\]
This is nonzero by Lemma \ref{lem:vphii-1i+1}, 
so $f_1$, $f_2$, $f_3$ are all $0$.
We have shown that  $I-J$, $A(I-J)$, $B(I-J)$, $ABC(I-J)$ are linearly independent.
By Lemma \ref{lem:mainB} $\dim {\cal X} \leq 8$.
The result follows. 
\end{proofof}

\section{Acknowledgement}

The author thanks Paul Terwilliger for many insightful comments that lead to
great improvements in the paper.

\section{Appendix 1}

Assume $d \geq 3$.
Let $A,B,C$ be a nonbipartite normalized LR triple
that is not $q$-Weyl type.
By Theorem \ref{thm:main2} the elements
\[
   I, \quad
   A,  \quad
   B, \quad
   C, \quad
   ABC, \quad
   ACB, \quad
   CAB
\]
form a basis for the tridiagonal space $\cal X$.
We represent the elements $BCA$, $BAC$, $CBA$ as a linear combination of
the basis vectors.
Below we give the coefficients of the linear combination.

For $\text{\rm NBG}_d(\F;q)$:
\[
\begin{array}{c|ccccccc}
 & I & A & B & C & ABC & ACB & CAB
\\ \hline
 BCA     
&           \rule{0mm}{6mm}
  \frac{(q^d-1)(q^{d+2}-1)}
         {q^d (q-1)^2} 
&
  - \frac {q^2-1}
            {q-1} 
&
 - \frac{ q^2-1}
           { q-1 } 
&
  \frac{ (q^2-1)^2 }
           {q (q-1)^2 }
&
  \frac{ q^2-1 }
           {q (q-1) }
&
 - (2q+1) 
&
  \frac{ q^2-1 }
           {q (q-1)}
\\
 BAC                \rule{0mm}{5mm}
&
  \frac{(q^d-1)(q^{d+2}-1)}
         {q^{d+1} (q-1)^2}
&
 0
&
 - \frac{ q^2-1}
           { q(q-1) } 
&
  \frac{ q^2-1 }
           {q^2 (q-1) }
&
  \frac{ q^2-1 }
           {q^2 (q-1) } 
&
 - \frac{q^2-1}
           {q (q-1) }
&
  \frac{1 }
           {q^2}
\\
 CBA                      \rule{0mm}{5mm}
&
  \frac{(q^d-1)(q^{d+2}-1)}
         {q^{d+1} (q-1)^2} 
&
 - \frac{ q^2-1}
           { q(q-1) }
&
 0
&
  \frac{ q^2-1 }
           {q^2 (q-1) }
&
  \frac{1 }
           {q^2 }
&
 - \frac{q^2-1}
           {q (q-1) }
&
  \frac{q^2-1 }
           {q^2(q-1)}
\end{array}
\]

For $\text{\rm NBG}_d(\F;1)$:
\[
\begin{array}{c|ccccccc}
 & I & A & B & C & ABC & ACB & CAB
\\ \hline
 BCA     \rule{0mm}{5mm}
&
 d(d+2)
&
 - 2 
&
 - 2
&
  4
&
  2 
&
 - 3
&
 2 
\\
 BAC   \rule{0mm}{5mm}
&
  d(d+2)
&
 0
&
 - 2 
&
  2 
&
  2 
&
 - 2 
&
 1 
\\
 CBA    \rule{0mm}{5mm}
&
  d(d+2) 
&
 - 2 
&
 0
&
  2 
&
  1
&
 - 2
&
  2
\end{array}
\]

For $\text{\rm NBNG}_d(\F;t)$:
\[
\begin{array}{c|ccccccc}
 & I & A & B & C & ABC & ACB & CAB
\\ \hline
 BCA                 \rule{0mm}{6mm}
&
 \frac{(t^{d/2} -1)(t^{(d+2)/2}-1) }
        {t^{d/2} } 
&
 t-1 
&
 t-1 
&
 0
&
 0
&
  t 
&
 0
\\
 BAC            \rule{0mm}{6mm}
&
 - \frac{(t^{d/2} -1)(t^{(d+2)/2}-1) }
        {t^{(d+2)/2} } 
&
 0
&
  - \frac{t-1}
            { t } 
&
   - \frac{t-1}
            { t } 
&
 0
& 
 0
&
    \frac{ 1} {t} 
\\
 CBA            \rule{0mm}{6mm}
&
  - \frac{(t^{d/2} -1)(t^{(d+2)/2}-1) }
        {t^{(d+2)/2} }
&
  - \frac{t-1}
            { t } 
&
 0
&
   - \frac{t-1}
            { t } 
&
   \frac{ 1} {t}
&
 0
&
 0
\end{array}
\]

\section{Appendix 2}

Let $A,B,C$ be a bipartite normalized LR triple
with diameter $d \geq 4$.
Let $\cal X$ be the tridiagonal space for $A,B,C$.
By Theorem \ref{thm:main1},
the vector space ${\cal X} J$ has a basis
\[
  J, \quad
  A J, \quad
  B J, \quad
  ACB J.
\]
We represent the elements
\[
  CJ, \quad
  ABC J, \quad
  BACJ, \quad
  BCAJ, \quad
  CABJ, \quad
  CBAJ
\]
as a linear combination of the basis vectors.
Below we give the coefficients of the linear combination.

For $\text{\rm B}_d(\F;t, \rho_0, \rho'_0, \rho''_0)$:
\[
\begin{array} {c | c c cc}
 & J &A J & B J & ACB J 
\\ \hline
C J &   0 &
  \rho''_0   \rule{0mm}{5mm}
&    
  \frac{ t }
          { \rho'_0 }
&   
  \frac{ \rho''_0 (t-1) }
           {\rho'_0 } 
\\
A B C J &  0 &      \rule{0mm}{5mm}
  \frac{ \rho_0 \rho''_0 (t^{d/2} -1) }
          { t-1 } 
&
 0
&      \rule{0mm}{5mm}
  - \frac{ \rho_0 \rho''_0  }
           { \rho'_0 }
\\
B A C J &   0 &   \rule{0mm}{5mm}
- \frac{ \rho''_0 }
        { \rho_0 }
&    
  \frac{ \rho''_0 ( t^{d/2+1} - 1) }
           { t-1 } 
&    
 - \frac{\rho''_0 t }
           { \rho_0 \rho'_0 }
\\
B C A J &  0 &    \rule{0mm}{5mm}
  \rho'_0 
&   
  \frac{ t }
         { \rho''_0 }
&
  t 
\\
C A B J &   0 &    \rule{0mm}{5mm}
 0
&   
  \frac{ \rho''_0 ( t^{d/2+1} - 1) }
           {  t-1 }
&    
 - \frac{\rho''_0 t }
          {\rho_0 \rho'_0 } 
\\
 C B A J &   0 &     \rule{0mm}{5mm}
  \frac{ \rho_0 \rho''_0 (t^{d/2} - 1) }
           { t-1 } 
&   
 - \frac{ \rho_0 }
           { \rho'_0 }
&     
 - \frac{ \rho_0 \rho''_0 }
           { \rho'_0 }
\end{array}
\]

For $\text{\rm B}_d(\F; 1, \rho_0, \rho'_0, \rho''_0)$:
\[
\begin{array} {c | c c cc }
 &J & A J & B J & ACB J 
\\ \hline
C J
 &    0 &
  \rho''_0   \rule{0mm}{5mm}
&    
  \frac{ 1 }
          { \rho'_0 }
&    
 0
\\
A B C J 
&     0 &    \rule{0mm}{5mm}
 - \frac{ d }
          { 2 \rho'_0 } 
&
 0
&         \rule{0mm}{5mm}
   \frac{ 1 }
           { {\rho'_0}^2  }
\\
B A C J &    0 &    \rule{0mm}{5mm}
- \frac{ \rho''_0 }
        { \rho_0  }
&    
  \frac{ \rho''_0 (d+2)}
         {2}
&   
  {\rho''_0}^2
\\
B C A J &   0 &   \rule{0mm}{5mm}
  \rho'_0 
&    
  \frac{ 1 }
         { \rho''_0 }
&
  1
\\
C A B J &  0 &     \rule{0mm}{5mm}
 0
&   
  \frac{ \rho''_0 (d + 2) }
           { 2 }
&     
  {\rho''_0}^2
\\
 C B A J &   0 &    \rule{0mm}{5mm}
 - \frac{ d }
           {2 \rho'_0 } 
&  
 - \frac{ \rho_0 }
           { \rho'_0 }
&    
  \frac{ 1 }
           { {\rho'_0}^2 }
\end{array}
\]

By Theorem \ref{thm:main1},
the vector space ${\cal X} (I-J)$ has a basis
\[
  I-J, \quad
  A (I-J), \quad
  B (I-J), \quad
  ABC (I-J).
\]
We represent the  elements
\[
  C(I-J), \;
  ACB (I-J), \;
  BAC(I-J), \;
  BCA (I-J), \;
  CAB (I-J), \;
  CBA (I-J)
\]
as a linear combination of the basis vectors.
Below we give the coefficients of the linear combination.

For $\text{\rm B}_d(\F;t, \rho_0, \rho'_0, \rho''_0)$:
\[
\begin{array}{c|cccc}
 & I-J & A (I-J) & B (I-J) & ABC (I-J)
\\ \hline
C (I-J)                   \rule{0mm}{5mm}
&  0 &
 - \frac{ \rho_0 \rho'_0 }
           { t }
&    
  \rho'_0 
&  
 - \frac{ \rho_0 (t-1) }
           { t } 
\\
A C B (I-J)            \rule{0mm}{5mm}
&   0 &
  \frac{ \rho_0 \rho''_0 (t^{d/2+1} - 1) }
          {t (t-1) } 
&
 0
&     
 - \frac{ \rho_0 \rho''_0 }
           { \rho'_0 t } 
\\
B A C (I-J)           \rule{0mm}{5mm}
&    0 &
 - \frac{ \rho_0^2 \rho'_0 (t^{d/2} - 1 ) }
           { t (t-1) } 
&    
  \frac{ \rho_0 \rho'_0 (t^{d/2}-1) }
           { t-1 }
&   
  \frac{ \rho_0^2 }
           { t }
\\
B C A (I-J)           \rule{0mm}{5mm}
&   0 &
   \frac{ \rho_0 \rho''_0 (t^{d/2} - 1) }
           { t-1 } 
&   
  \rho''_0
&  
 -  \frac{  \rho_0 \rho''_0 }
           { \rho'_0 } 
\\
C A B (I-J)            \rule{0mm}{5mm}
&  0 &
 - \frac{ \rho_0^2 \rho'_0 (t^{d/2+1} - 1) }
           { t (t-1) }
&   
  \frac{\rho_0 \rho'_0 (t^{d/2} - 1) }
          { t-1 }
&   
  \frac{ \rho_0^2 }
           { t } 
\\
C B A (I-J)           \rule{0mm}{5mm}
& 0 &
 0 
&   
 - \frac{ \rho'_0 }
          { \rho_0 }
&
  1
\end{array}
\]

For $\text{\rm B}_d(\F;1, \rho_0, \rho'_0, \rho''_0)$:
\[
\begin{array}{c|cccc}
 & I-J & A (I-J) & B (I-J) & ABC (I-J)
\\ \hline
C (I-J)                   \rule{0mm}{5mm}
&   0 &
  \frac{ 1 }
         { \rho''_0 }
&    
  \rho'_0 
&   
 0
\\
A C B (I-J)            \rule{0mm}{5mm}
&    0 &
 - \frac{ d+2 }
          { 2 \rho'_0 } 
&
 0
&   
  \frac{ 1 }
           { {\rho'_0}^2  } 
\\
B A C (I-J)           \rule{0mm}{5mm}
&      0 &
  \frac{ d \rho_0 }
           { 2 \rho''_0 } 
&   
 -  \frac{ d  }
           { 2 \rho''_0 }
&   
 \rho_0^2
\\
B C A (I-J)           \rule{0mm}{5mm}
&   0 &
 - \frac{ d }
           {2 \rho'_0 } 
&   
  \rho''_0
&   
  \frac{ 1 }
           { {\rho'_0}^2 } 
\\
C A B (I-J)            \rule{0mm}{5mm}
&   0 &
  \frac{ \rho_0 (d+2) }
           { 2 \rho''_0 }
&   
  -\frac{ d }
          { 2 \rho''_0 }
&  
 \rho_0^2
\\
C B A (I-J)           \rule{0mm}{5mm}
& 0 &
 0 
&   
 - \frac{ \rho'_0 }
          { \rho_0 }
&
  1
\end{array}
\]

\bigskip\bigskip\noindent
Kazumasa Nomura\\
Tokyo Medical and Dental University\\
Kohnodai, Ichikawa, 272-0827 Japan\\
email: knomura@pop11.odn.ne.jp

\medskip\noindent
{\small
{\bf Keywords.} Lowering map, raising map, quantum group, quantum algebra, Lie algebra
\\
\noindent
{\bf 2010 Mathematics Subject Classification.} 17B37, 15A21
}

\end{document}